\documentclass[reqno,10pt,centertags]{amsart}
\usepackage{amsmath,amsthm,amscd,amssymb,latexsym,esint,upref,stmaryrd,
enumerate,color,verbatim,yfonts}
\usepackage{hyperref}
\newcommand*{\mailto}[1]{\href{mailto:#1}{\nolinkurl{#1}}}

\newcommand{\bbC}{{\mathbb{C}}}

\newcommand{\bbN}{{\mathbb{N}}}

\newcommand{\bbR}{{\mathbb{R}}}
\newcommand{\bbS}{{\mathbb{S}}}

\newcommand{\bbZ}{{\mathbb{Z}}}

\newcommand{\cB}{{\mathcal B}}

\newcommand{\cH}{{\mathcal H}}

\newcommand{\beq}{\begin{equation}}
\newcommand{\enq}{\end{equation}}

\DeclareMathOperator{\ran}{ran}
\DeclareMathOperator{\dom}{dom}

\DeclareMathOperator{\ind}{ind}

\DeclareMathOperator*{\sgn}{sgn}
\DeclareMathOperator*{\defe}{def}

\renewcommand{\Im}{\text{\rm Im}}
\renewcommand{\ln}{\text{\rm ln}}

\newcommand{\no}{\notag}
\newcommand{\lb}{\label}
\newcommand{\f}{\frac}

\newcommand{\ol}{\overline}

\newcommand{\wti}{\widetilde}
\newcommand{\Oh}{O}

\newcommand{\hatt}{\widehat}
\newcommand{\dott}{\,\cdot\,}

\renewcommand{\dot}{\overset{\textbf{\Large.}}}
\renewcommand{\ddot}{\overset{\textbf{\Large..}}}
\renewcommand{\dotplus}{\overset{\textbf{\Large.}} +}

\newcommand{\bi}{\bibitem}

\let\geq\geqslant
\let\leq\leqslant

\makeatletter
\def\theequation{\@arabic\c@equation}

\allowdisplaybreaks
\numberwithin{equation}{section}

\newtheorem{theorem}{Theorem}[section]
\newtheorem{proposition}[theorem]{Proposition}
\newtheorem{lemma}[theorem]{Lemma}
\newtheorem{corollary}[theorem]{Corollary}
\newtheorem{definition}[theorem]{Definition}
\newtheorem{hypothesis}[theorem]{Hypothesis}
\newtheorem{example}[theorem]{Example}

\theoremstyle{remark}
\newtheorem{remark}[theorem]{Remark}

\begin{document}

\title[On the Product Formula for Defect Numbers of Closed Operators]{Some Remarks on the Product Formula for Defect Numbers of Closed Operators}
\author[C.\ Fischbacher]{Christoph Fischbacher}
\address{Department of Mathematics, Baylor University, Sid Richardson Bldg., 1410 S.\,4th Street, Waco, TX 76706, USA}
\email{\mailto{C\_Fischbacher@baylor.edu}}
\urladdr{\url{https://math.artsandsciences.baylor.edu/person/christoph-fischbacher-phd}}

\author[F.\ Gesztesy]{Fritz Gesztesy}
\address{Department of Mathematics, Baylor University, Sid Richardson Bldg., 1410 S.\,4th Street, Waco, TX 76706, USA}
\email{\mailto{Fritz\_Gesztesy@baylor.edu}}
\urladdr{\url{https://math.artsandsciences.baylor.edu/person/fritz-gesztesy-phd}}

\author[L.\ L.\ Littlejohn]{Lance L. Littlejohn}
\address{Department of Mathematics, Baylor University, Sid Richardson Bldg., 1410 S.\,4th Street, Waco, TX 76706, USA}
\email{\mailto{Lance\_Littlejohn@baylor.edu}}
\urladdr{\url{https://math.artsandsciences.baylor.edu/person/lance-littlejohn-phd}}

\dedicatory{Dedicated with great affection and admiration to Henk de Snoo \\ on the happy occasion of his 80th birthday}

\date{\today}
\@namedef{subjclassname@2020}{\textup{2020} Mathematics Subject Classification}

\begin{abstract}
This largely pedagogical note recalls some facts on defect numbers of products of closed operators employing results from the theory of semi-Fredholm operators and then applies these facts to positive integer powers of symmetric operators and subsequently to certain minimal Sturm--Liouville and minimal higher even-order ordinary differential operators. We also point out some unexpected missed opportunities when comparing the work of different groups on this subject.  
\end{abstract}

\maketitle

{\scriptsize{\tableofcontents}}

\section{Introduction} \lb{s1}

A very happy birthday, Henk! We hope this note will cause some smiles on your part.

This primarily pedagogical note originated with the following explicit question: ``What are the deficiency indices of positive integer powers of of the minimal Legendre operator in $L^2((-1,1))$?'' It did not take long to realize that this question, actually, a much more general one, had long been answered by at least three different groups. 

In 1950 (Engl. translation in 1953) I.~M.~Glazman \cite{Gl53} determined the deficiency indices of products of densely defined and closed linear operators $T_1$ and $T_2$ in a complex separable Hilbert space assuming finite defect numbers of $T_1, T_2$, that is,
\begin{equation}
\defe(T_j,0) = \dim(\ker(T_j^*)) < \infty, \quad j=1,2,
\end{equation}
and $0$ in the domain of regularity of $T_1, T_2$, that is, there exists constants $c_{T_j} \in (0,\infty)$ such that 
\begin{equation}
\|T_j f\|_{\cH} \geq c_{T_j} \|f\|_{\cH}, \quad f \in \dom(T_j), \; j=1,2.
\end{equation}
Under these hypotheses, Glazman proved that 
\begin{equation}
\defe(T_2 T_1,0) = \defe(T_1,0) + \defe(T_2,0). 
\end{equation}
In addition, Glazman proved the result that if $S$ is symmetric with with finite deficiency indices
\begin{equation}
n_{\pm}(S) = \dim(\ker(S^* \mp i I_{\cH})) \in \bbN,
\end{equation}
then
\begin{equation}
n_{\pm}\big(S^2\big) = n_+(S) + n_-(S)     \lb{1.5} 
\end{equation}
(These results extend of course to higher positive integer powers of $S$.) These 1950 results by Glazman are also mentioned in his 1965 book \cite[p.~24--27]{Gl65}. 

R.~M.~Kauffman, T.~Read, and A.~Zettl, in their 1977 lecture notes \cite[Ch.~V.4]{KRZ77} on the deficiency indices of powers of ordinary differential expressions are citing Glazman's 1950 paper, but not his 1965 book. 

At the same time, H.~Behncke and H.~Focke also derived \eqref{1.5} (and its higher positive integer power analog) in 1977 as a byproduct of their study of stability properties of of deficiency indices with respect to sufficiently small relative bounded perturbations. In addition, Behncke and Focke also deal with the case of possibly infinite deficiency indices. Interestingly, Behncke and Focke do not cite Glazman and, being essentially simultaneous with Kauffman, Read, and Zettl, neither group appears to have been aware of the activities of the other. 

In addition to the three groups mentioned, there has been considerable interest in studying powers and generally, products of symmetric operators, both, in an abstract context (see, e.g., \cite{AKK95}, \cite{Ar23}, \cite{Ar24}, \cite[App.~D]{BHS20}, \cite{BN93}, \cite{DM23}, \cite{HKM00}, \cite{Sc83}) and particularly in the context of (minimally) defined differential operators (see, e.g., \cite{AS04}, \cite{CSE99}, \cite{Ch73}, \cite{Ev80}, \cite{EG71}, \cite{EG72}, \cite{EG72a}, \cite{EG74/75}, \cite{EG75}, \cite{EKLWY07}, \cite{ELW02}, \cite{EZ77}, \cite{EZ78}, \cite{HKM00}, \cite{Ka73}, \cite{Ka75}, \cite[Ch.~V]{KRZ77}, \cite{Ro84}, \cite{Ze74}, \cite{Ze75a}, \cite{Ze75}, \cite{Ze76}). For fundamental papers on the number of $L^2$-solutions associated with $n$th-order differential equations, $n \in \bbN)$ (including the case of matrix-valued coefficients), see, for instance,  \cite{Ev59}, \cite{KR73/74}, \cite{KR74/75}, \cite{LM01}, \cite{LM03}, \cite{Mi01}, \cite{Or53},  \cite{Ze76a}, and the literature cited therein. 

In Section \ref{s2} we recall the basic abstract results of defect numbers of products of densely defined closed operators and, utilizing the notion of semi-Fredholm operators, apply this to positive integer powers of symmetric operators $S$ in $\cH$ in Theorems \ref{t2.8} and \ref{t2.9}. This is then extended to real polynomials of $S$, recovering a result of Behncke and Focke (and extending the corresponding results of Glazman). We also provide a discussion of the product formula for (semi-)Fredholm operators and in this context note that Sandovici and de Snoo describe the index formula for products of linear relations in \cite{SdS09}. 

In Section \ref{s3} we hint at some possible applications to positive integer powers of certain minimal Sturm--Liouville and minimal higher even-order ordinary differential operators.  In addition, we analyze in great detail the differential expression introduced by J.~Chaudhuri and W.~N.~Everitt \cite{CE69} 
\begin{equation}
\tau_{2,CE} = - \f{d}{dx} \f{1}{6} (x+1)^4 \f{d}{dx} + (x+1)^2, \quad x \in [0,\infty),    \lb{1.6}
\end{equation}
which we reduce to the special case $\alpha = \sqrt{33}/2 \in (1,3)$ of the following one-parameter Bessel differential expression
\begin{equation}
\tau_{2,\alpha} = - \f{d^2}{dx^2} + \f{\alpha^2 - (1/4)}{(1-x)^2}, \quad \alpha \in [1,\infty), \; x \in [0,1),   \lb{1.7} 
\end{equation}
implying 
\begin{align}
\begin{split} 
\tau_{4,\alpha} = \tau_{2,\alpha}^2 = \f{d^4}{dx^4} - \f{d}{dx} \f{2 \alpha^2 - (1/2)}{(1-x)^2} \f{d}{dx} 
+ \f{\alpha^4 - (13/2) \alpha^2 + (5/4)^2}{(1-x)^4},&     \lb{1.8} \\
\alpha \in [1,\infty), \; x \in [0,1).&
\end{split}
\end{align}
The point of this example lies in the facts 
\begin{align}
& \defe(T_{2,\alpha,min},0) = 1,     \lb{1.9} \\
& \defe\big(T_{2,\alpha,min}^2,0\big) = 2 = 2 \defe(T_{2,\alpha,min},0),      \lb{1.10} \\
& \defe(T_{4,\alpha,min},0) = 3; \quad \alpha \in [1,3),  \lb{1.11}
\end{align}
where \eqref{1.10} is of course in agreement with \eqref{1.5}, and \eqref{1.11} is rather remarkable. Here $T_{2,\alpha,min}$ and $T_{4,\alpha,min}$ denote the minimal operators associated with the differential expressions $\tau_{2,\alpha}$ and $\tau_{4,\alpha}$ in $L^2((0,1))$. The property $\defe(T_{4,\alpha,min},0) = 3$ in \eqref{1.11} is dubbed the ``limit-$3$ property'' of $T_{4,\alpha,min}$ in \cite{CE69} and we refer, for instance, also to \cite{EG72}, \cite{EG72a}, \cite{EG74/75}, \cite{EG75}, \cite{EHW74}, 
for more details in this connection. 

We conclude Section \ref{s3} with some applications to partial differential operators. First, we study strongly singular, homogenous perturbations of the Laplacian on $\bbR^n$ of the form 
\begin{equation}
\big[-\Delta_n - \{[(n-1)(n-3)/4] + L(L+n-2)\} |x|^{-2}\big]\big|_{C_0^\infty(\bbR^n\backslash\{0\}}.     \lb{1.12} 
\end{equation}
It is well-known that this operator is unitarily equivalent to a direct sum of countably many Bessel differential operators, where the first $L+1$ members of this direct sum have deficiency indices both equal to one, while the remaining members are essentially self-adjoint and therefore have deficiency indices both equal to zero. From this, we conclude that the deficiency indices of (the closure of) the operator in \eqref{1.12} are equal to $L+1$ and hence, its $m^{th}$ power has deficiency indices  $m(L+1)$.

Finally, we consider singular perturbations of the Dirichlet Laplacian $T_{\Omega, D}$ in $L^2(\Omega)$, where $\Omega\subset \bbR^n$ is a bounded domain with smooth boundary $\partial\Omega$. By ``singular perturbations", we mean symmetric restrictions $T_{\Omega,h,k}$ of $T_{\Omega, D}$ with defect indices equal to one, which are parametrized by $h\in L^2(\Omega)$ and $k\in C(\partial\Omega)$. The domains of these restrictions are described by the additional non-local boundary condition 
\begin{equation} 
\dom(T_{\Omega,h,k})=\{g\in\dom(T_{\Omega,D})\, | \, (h,g)_{L^2(\Omega)}=(k,\gamma_N g)_{L^2(\partial\Omega)}\},
\end{equation}  
where $\gamma_N$ is the Neumann trace on the boundary $\partial\Omega$. We determine the adjoint $T_{\Omega,h,k}^*$ and the Friedrichs and Krein--von Neumann extensions of $T_{\Omega,h,k}$. Moreover, using von Neumann's theory of self-adjoint extensions of symmetric operators, we describe the one-parameter family of all self-adjoint extensions of $T_{\Omega,h,k}$ and, using the Birman--Krein--Vishik theory, we also describe all of its nonnegative self-adjoint extensions. Applying the general results from Section \ref{s2}, we also get a description of the powers $T_{\Omega,h,k}^m$ and their defect indices. 

Finally, we briefly summarize some of the notation used in this paper: Let $\cH$ be a separable
complex Hilbert space, $(\cdot,\cdot)_{\cH}$ the scalar product in
$\cH$ (linear in the second factor), and $I_{\cH}$ the identity operator in $\cH$.
Next, let $T$ be a linear operator mapping (a subspace of) a
Banach space into another, with $\dom(T)$, $\ran(T)$, and $\ker(T)$
denoting the domain, range, and kernel (i.e., null space) of $T$.

The spectrum
and resolvent set of a closed linear operator $T$ in $\cH$ will be denoted by $\sigma(T)$
and $\rho(T)$, respectively; the field of regularity of a linear operator $T$ is denoted by $\hatt \rho(T)$, and if $T$ is self-adjoint, it's essential spectrum is denoted by $\sigma_{ess}(T)$. 

The Banach space of bounded 
linear operators on $\cH$ is denoted by $\cB(\cH)$, the corresponding two-Hilbert space situation is 
abbreviated by $\cB(\cH_1,\cH_2)$. 


To simplify notation, we will write $L^2(\Omega)$ instead of $L^2(\Omega; d^nx)$, where $\Omega \subseteq \bbR^n$, $n \in \bbN$, whenever the underlying Lebesgue measure is understood. 

The open upper and lower complex half-planes are abbreviated by $\bbC_{\pm} = \{z \in \bbC \, | \, \pm \Im(z) \in (0,\infty)\}$, and we use the notation $\bbN_0 = \bbN \cup \{0\}$.

\section{On a Formula of Glazman for the Defect Number of Products of Closed Operators} \lb{s2}

In this section we discuss and slightly extend Glazman's formula for the defect number of products of densely defined, closed, operators with $0$ in their domain of regularity (see \cite[\S~2]{Gl53}, \cite[Theorem~21, p.~26]{Gl65}) and then apply it to the deficiency indices of polynomials with real coefficients of a densely defined, closed, symmetric operator. 

For the remainder of Section \ref{s2} we make the following assumptions:

\begin{hypothesis} \lb{h2.1}
All Hilbert spaces $\cH$, $\cH_j$, $j=1,2,3,\dots$, are assumed to be complex and separable.   \\[1mm]
\end{hypothesis}

We start by collecting a number of well-known basic facts that will be useful in the remainder of this note (see, e.g., \cite[Sects.~1.3, 3.1, 3.2]{EE18}, \cite[Sects.~1--3]{Sc12} for details): In the following, $T$ is a linear operator in the separable, complex Hilbert space $\cH$: \\[1mm]
$(1)$ $z \in \bbC$ is called a regular point for $T$ if there exists $c_{z,T} \in (0,\infty)$ such that 
\begin{equation}
\|(T - z I_{\cH})f\|_{\cH} \geq c_{z,T} \|f\|_{\cH}, \quad f \in \dom(T).    \lb{2.1} 
\end{equation}
The set of regular points of $T$ is called the {\it field of regularity} $($also, the {\it regularity domain$)$\,of $T$} and denoted by $\hatt \rho(T)$. Moreover, $\hatt \rho(T)$ is open. \\[1mm]
$(2)$ $z \in \hatt \rho(T)$, then $\ran(T - z I_{\cH})^{\bot}$ is called the {\it deficiency subspace of $T$ at $z$}; its dimension,
\begin{equation}
\defe(T,z) = \defe(T-zI_{\cH},0) = \dim\big((\ran(T - z I_{\cH}))^{\bot}\big)
\end{equation} 
is called the {\it defect number of $T$ at $z$}. \\[1mm] 
$(3)$ If $T$ closable in $\cH$, then \\[1mm]
\hspace*{4.5mm} $(a)$ $\defe(T,\dott)$ is constant on each connected component of $\hatt \rho(T)$. \\[1mm] 
\hspace*{4.5mm} $(b)$ One has 
\begin{equation} 
\hatt \rho\big({\ol T}\big) = \hatt \rho(T), \quad d\big({\ol T},z\big) = d(T,z), \; z \in \hatt \rho(T), 
\end{equation}
\hspace*{10.5mm} and 
\begin{equation}
\ran\big({\ol T} - z I_{\cH}\big) = \ol{\ran(T - z I_{\cH})}, \quad z \in \hatt \rho(T).
\end{equation}
\hspace*{10.5mm} In particular, if $T$ closed (i.e., $T= {\ol T}$), $z \in \hatt \rho(T)$, then $\ran(T - z I_{\cH})$ \\
\hspace*{10.5mm} is a closed linear subspace of $\cH$. \\[1mm]
\hspace*{4.5mm} $(c)$ in addition, if $T$ is densely defined in $\cH$, then 
\begin{equation}
{\ol T} = (T^*)^*, 
\end{equation}
\hspace*{10.5mm} and 
\begin{align}
\cH &= \ol{\ran(T - z I_{\cH})} \oplus \ker(T^* - {\ol z} I_{\cH}), \quad z \in \bbC,  \\
&=  \ol{\ran(T^* - {\ol z} I_{\cH})} \oplus \ker(T - z I_{\cH}), \quad z \in \bbC,  \\
&= \ran\big({\ol T} - z I_{\cH}\big) \oplus \ker(T^* - {\ol z} I_{\cH}), \quad z \in \hatt \rho(T).
\end{align}
$(4)$ If $T$ is closed, then
\begin{align}
\begin{split} 
\rho(T) &= \{z \in \hatt \rho(T) \, | \, \defe(T,z) = 0\}     \\
&= \{z \in \hatt \rho(T) \, | \, \ran(T - z I_{\cH}) = \cH\} \subseteq \hatt \rho(T),   
\end{split} \\
\sigma(T) &= \bbC \backslash \rho(T).
\end{align}
$(5)$ Let $T$ be densely defined in $\cH$ and closed. \\[1mm]  
\hspace*{4.5mm} $(a)$ If $z \in \bbC$, then $\ran(T - z I_{\cH})$ is closed in $\cH$ if and only if 
$\ran(T^* - {\ol z} I_{\cH})$ is \\
\hspace*{10.5mm} closed in $\cH$.  \\
\hspace*{4.5mm} $(b)$ If $z \in \hatt \rho(T)$, then $\ran(T - z I_{\cH})$ and $\ran(T^* - {\ol z} I_{\cH})$ are closed in $\cH$. \\[1mm] 
$(6)$ If $T$ is densely defined in $\cH$, then 
\begin{equation}
(\ran(T - zI_{\cH}))^{\bot} = \ker(T^* - {\ol z} I_{\cH}), \quad z \in \bbC.
\end{equation}
$(7)$ If $T$ is densely defined and closable in $\cH$, then 
\begin{equation}
(\ran(T^* - {\ol z} I_{\cH}))^{\bot} = \ker\big({\ol T} - z I_{\cH}\big), \quad z \in \bbC.
\end{equation} 
$(8)$ If $T$ is densely defined in $\cH$, then $T$ is called {\it symmetric} (resp., {\it self-adjoint}) if $T \subseteq T^*$ (resp., $T=T^*$). \\[1mm]
$(9)$ Let $T$ be symmetric in $\cH$.  \\[1mm]
\hspace*{4.5mm} $(a)$ Then $T$ is closable in $\cH$. \\[1mm] 
\hspace*{4.5mm} $(b)$ $\bbC \backslash \bbR \subseteq \hatt \rho(T)$. \\[1mm]
\hspace*{4.5mm} $(c)$ If $z \in \hatt \rho(T)$, then $\ran\big(T^* - {\ol z} I_{\cH}\big) = \cH$ and $\ker\big({\ol T} - z I_{\cH}\big) = \{0\}$. \\[1mm]   
\hspace*{4.5mm} $(d)$ If for some $c \in \bbR$, $T \geq c I_{\cH}$ (resp., $T \leq c I_{\cH}$), then 
$(-\infty,c) \subset \hatt \rho(T)$ (resp., \\
\hspace*{10mm} $(c,\infty) \subset \hatt \rho(T)$). \\[1mm] 
\hspace*{4.5mm} $(e)$ If $\hatt \rho(T)$ contains a real number (e.g., if $T$ is bounded from below or bounded  \\ 
\hspace*{10mm} from above), then $\hatt \rho(T)$ is connected and $\defe(T,\dott)$ is constant on $\hatt \rho(T)$.  
\\[1mm]
$(10)$ If $0 \in \hatt \rho(T)$ (with $\|Tf\|_{\cH} \geq c_{0,T} \|f\|_{\cH}$, $f \in \dom(T)$), then $0 \in \hatt \rho\big(T^2\big)$ and \\
\hspace*{6.5mm} $T^2 \geq c_{0,T}^2 I_{\cH}$. \\[1mm] 
$(11)$ If $T_2 : \dom(T_2) \to \cH_3$, $\ol{\dom(T_2)} = \cH_2$,  and $\ol{\dom(T_2 T_1)} = \cH_1$, then $T_1$ is densely defined in $\cH_1$ and 
\begin{equation} 
T_1^* T_2^* \subseteq (T_2 T_1)^*.
\end{equation}
If in addition, $T_2 \in \cB(\cH_2,\cH_3)$, then $(T_2 T_1)^*$ exists and 
\begin{equation} 
T_1^* T_2^* = (T_2 T_1)^*.     \lb{2.14}
\end{equation}
Similarly, if $T_1: \dom(T_1) \to \cH_2$, $\ol{\dom(T_1)} = \cH_1$ and $\dim(\ker(T_1^*)) < \infty$ (i.e., if $T_1$ is right semi-Fredholm), and if $T_2: \dom(T_2) \to \cH_3$, $\ol{\dom(T_2)} = \cH_2$, then $(T_2 T_1)^*$ exists and \eqref{2.14} holds (cf.\ Proposition \ref{p3.1}).

In this context we recall the standard convention,
\begin{equation}
\dom(T_2T_1) = \{g \in \dom(T_1) \,|\, T_1 g \in \dom(T_2)\}.  
\end{equation}

Since operators with closed range will play a particular role in the following, we mention the following criterion:

\begin{lemma} $($\cite[Lemma~A.1]{ASS94}, \cite[Lemma~2.44]{BB13}.$)$ \lb{l2.2}  ${}$ \\
Suppose $T\colon \dom(T) \to \cH_2$, $\dom(T) \subseteq \cH_1$, is densely defined and closed. Then the following items $(i)$--$(iii)$ are equivalent: \\[1mm]
$(i)$ \quad $\ran(T)$ is closed in $\cH_2$. \\[1mm]
$(ii)$ \; $0 \notin \sigma(T^*T)$, or, $0$ is an isolated point in $\sigma(T^*T)$. \\
$(iii)$ There exists $c_T \in (0,\infty)$ such that \\[1mm]    
\hspace*{7mm}  $\|T f\|_{\cH} \geq c_T \|f\|_{\cH}$, \, $f \in \dom(T) \cap [\ker(T)]^{\bot}$.   \\[1mm] 
In particular, $\ran(T)$ is closed if $0 \in \hatt \rho(T)$. 
\end{lemma}

This can be contrasted with \eqref{2.1} for $z=0$. 

Next, we recall the notion of Fredholm and left semi-Fredholm operators and their index:

\begin{definition} \lb{d2.3}
Suppose $T:\dom(T) \to \cH_2$, $\dom(T) \subseteq \cH_1$, is densely defined and closed and $\ran(T)$ is closed in $\cH_2$. \\[1mm] 
$(i)$ Then $T$ is Fredholm if 
\begin{equation}
\dim(\ker(T)) + \dim(\ker(T^*)) < \infty.
\end{equation}
$(ii)$ Similarly, $T$ is a left semi-Fredholm operator if 
\begin{equation}
dim(\ker(T)) < \infty.
\end{equation}
$(iii)$ Finally, $T$ is a right semi-Fredholm operator if 
\begin{equation}
\dim(\ker(T^*)) < \infty.
\end{equation}

In all cases, the $($semi-$)$Fredholm index, $\ind(T)$, of $T$ is defined via 
\begin{equation}
\ind(T) = \dim(\ker(T)) - \dim(\ker(T^*)) \in \bbZ \cup \{-\infty, +\infty\}. 
\end{equation}
$($Of course, $\ind(T) \in \bbZ$ if $T$ is Fredholm.$)$
\end{definition}

One notes that under the hypothesis in Definition \ref{d2.3}, 
\begin{align}
\begin{split} 
& \text{$T$ is a (left semi-) Fredholm operator} \\
& \quad \text{if and only if} \\
& \text{$T^*$ is a (right semi-) Fredholm operator.} 
\end{split} 
\end{align}
In addition, since
\begin{equation}
\ker(T) = \ker(T^*T), \quad \ker(T^*) = \ker(TT^*),
\end{equation}
one gets
\begin{equation}
\ind(T) = \dim(\ker(T^*T)) - \dim(\ker(TT^*)) \in \bbZ \cup \{-\infty,+\infty\},  
\end{equation}
which represents a reduction to the case of self-adjoint and nonnegative operators $T^*T \geq 0$, and $TT^* \geq 0$, in $\cH_1$, and $\cH_2$, respectively. Moreover, we recall that 
\begin{equation}
\ind(T^*) = - \ind(T).
\end{equation}

In the special case where $T=T^*$ is self-adjoint in $\cH$ we note that \cite[p.~421--424]{EE18} implies 
\begin{align}
\begin{split} 
& T \, \text{ is Fredholm if and only if there exists $\varepsilon \in (0,\infty)$ such that} \\
& \quad \sigma_{ess}(T) \cap (-\varepsilon,\varepsilon) = \emptyset.     \lb{2.24} 
\end{split}
\end{align}

Next, we recall the following convenient criterion for $T$ to be a Fredholm operator in terms of the self-adjoint and nonnegative operators $T^*T$ and $TT^*$ in $\cH_1$ and $\cH_2$, respectively: Suppose $T:\dom(T) \to \cH_2$, $\dom(T) \subseteq \cH_1$, is densely defined and closed. Then (see, e.g., \cite[Appendix]{ASS94}, \cite[p.~724, 740]{GNZ24}, 
\begin{align}
\begin{split} 
& T \, \text{ is Fredholm if and only if there exists $\varepsilon \in (0,\infty)$ such that}     \lb{2.25} \\
& \quad \inf(\sigma_{ess}(T^*T)) \geq \varepsilon^2 \, \text{ and } \, \inf(\sigma_{ess}(TT^*)) \geq \varepsilon^2.  
\end{split} 
\end{align}
To illustrate the fact \eqref{2.25} one can argue as follows: Introduce the self-adjoint operator 
\begin{equation}
Q = \begin{pmatrix} 0 & T^* \\ T & 0 \end{pmatrix}, \quad \dom(Q) = \dom(T) \oplus \dom(T^*), 
\end{equation}
in $\cH_1 \oplus \cH_2$, and note that 
\begin{align}
& \ker(Q) = \ker(T) \oplus \ker(T^*), \quad \ran(Q) = \ran(T^*) \oplus \ran(T),    \lb{2.27} \\
& \sigma_3 Q \sigma_3 = - Q, \quad \sigma_3 = \begin{pmatrix} I_{\cH_1} & 0 \\ 0 & - I_{\cH_2} \end{pmatrix} 
= \sigma_3^* = \sigma_3^{-1},    \lb{2.28} \\
& Q^2 = \begin{pmatrix} T^*T & 0 \\ 0 & TT^* \end{pmatrix} = T^*T \oplus TT^*,   \lb{2.29}
\end{align}
in particular, $\sigma(Q)$, a closed subset of the real line, is symmetric with respect to the origin $0$.
 
Consequently, employing \eqref{2.24} and \eqref{2.27}--\eqref{2.29}, one obtains 
\begin{align}
& \text{$Q$ is Fredholm $\iff$ $T$ and $T^*$ are Fredholm}     \no \\
& \text{\quad $\iff$ $T$ is Fredholm $\iff$ $T^*$ is Fredholm}     \no \\
& \text{\quad \, $\Longrightarrow$ $Q^2$ is Fredholm $\iff$ $T^*T$ and $TT^*$ are Fredholm.}   \lb{2.30}
\end{align}
Similarlly,
\begin{align}
& \text{$T^*T$ and $TT^*$ are Fredholm $\iff$ there exists $\varepsilon \in (0,\infty)$ such that}    \no \\
& \text{\quad  $\sigma_{ess}(T^*T) \subseteq [\varepsilon^2,\infty)$ and $\sigma_{ess}(TT^*) \subseteq [\varepsilon^2,\infty)$ $\iff$ $Q^2$ is Fredholm}     \no \\
& \text{\quad $\iff$ there exists $\varepsilon \in (0,\infty)$ such that $\sigma_{ess}\big(Q^2\big) \subseteq [\varepsilon^2,\infty)$}    \no \\
& \text{\quad $\iff$ $\sigma_{ess}(Q) \subseteq (-\infty,-\varepsilon] \cup [\varepsilon,\infty)$ $\iff$ $Q$ is Fredholm}  
\no \\
& \text{\quad $\iff$ $T$ (and hence $T^*$) is Fredholm.}     \lb{2.31} 
\end{align}
Together, \eqref{2.30} and \eqref{2.31} prove \eqref{2.25}, among other facts. 

Although not needed in this context, we also recall, see \cite{De78},
\begin{align}
& \, \sigma(T^*T)\backslash\{0\} = \sigma(TT^*)\backslash\{0\},    \\
& \dim(\ker(T^*T - \lambda I_{\cH_1})) = \dim(\ker(TT^* - \lambda I_{\cH_1})), \quad \lambda \in \bbR\backslash\{0\}.
\end{align}
We refer to \cite{De78} (see also \cite{ASS94}, \cite[App.~A]{GGHT12}) for additional facts on $Q$, $T^*T$, $TT^*$. 

The celebrated product formula for the index of (left semi-) Fredholm operators then reads as follows:

\begin{theorem} \lb{t2.4} $($see, e.g., \cite{BF77}, \cite[Sect.~2.6]{BB13}, \cite{CL63}, \cite[Sect.~1.3]{EE18},  \cite{Gl53}, \cite[p.~24--27]{Gl65}, \cite[Chs.~XI, XVII]{GGK90}, \cite[Sect.~IV.2 ]{Go85}, \cite[Sect.~5.1]{Ku20}, \cite[Chs.~5, 7]{Sc02}$)$. ${}$ \\[1mm] 
$(i)$ Suppose $T_j$, $j=1,2$, are Fredholm operators $($hence, densely defined and closed\,$)$ with $\dom(T_1) \subseteq \cH_1$, $\ran(T_1) \subseteq \cH_2$, $\dom(T_2) \subseteq \cH_2$, $\ran(T_2) \subseteq \cH_3$. Then  $T_2 T_1$ is Fredholm $($hence, densely defined and closed\,$)$, and 
\begin{equation}
\ind(T_2 T_1) = \ind(T_1) + \ind(T_2) \in \bbZ. 
\end{equation}
$(ii)$ Suppose $T_j$, $j=1,2$, are left semi-Fredholm operators  $($hence, densely defined and closed\,$)$ with $\dom(T_1)\subseteq \cH_1$, $\ran(T_1) \subseteq \cH_2$, $\dom(T_2) \subseteq \cH_2$, $\ran(T_2) \subseteq \cH_3$, and assume that $T_2 T_1$ is densely defined in $\cH_1$. Then $T_2 T_1$ is left semi-Fredholm $($hence closed\,$)$, and 
\begin{equation}
\ind(T_2 T_1) = \ind(T_1) + \ind(T_2) \in \bbZ \cup \{-\infty\}.     \lb{2.35} 
\end{equation}
\end{theorem}

\begin{remark} \lb{r2.5}
$(i)$ Parts $(i)$ and $(ii)$ in Theorem \ref{t2.4} are proved in the special case of bounded (left and right semi-) Fredholm operators in \cite[Theorems~5.7, 5.26, 5.30]{Sc02}. Part $(i)$ in Theorem \ref{t2.4} for unbounded Fredholm operators is proved in \cite[Theorem~7.3]{Sc02}. Unbounded semi-Fredholm operators are discussed in \cite[Theorem~7.32]{Sc02}, except for the product formula \eqref{2.35}. However, replacing $T_1: \dom(T_1) \to \cH_2$, $T_2 : \dom(T_2) \to \cH_3$ by 
\begin{equation} 
\wti T_1 : \cH_{\wti T_1} \to \cH_2, \quad \wti T_2 : \cH_{\wti T_2} \to \cH_3,
\end{equation}
where
\begin{align}
\begin{split} 
& \cH_{\wti T_1} = \big(\dom\big(\wti T_1\big), (\dott,\dott)_{\wti T_1}\big),     \\
& (f_1,g_1)_{\wti T_1} = (T_1 f_1,T_1 g_1)_{\cH_2} + (f_1,g_1)_{\cH_1}, \quad f_1, g_1 \in \dom(T_1),   
\end{split} \\
\begin{split} 
& \cH_{\wti T_2} = \big(\dom\big(\wti T_2\big), (\dott,\dott)_{\wti T_2}\big),     \\
& (f_2,g_2)_{\wti T_2} = (T_2 f_2,T_2 g_2)_{\cH_3} + (f_2,g_2)_{\cH_2}, \quad f_2, g_2 \in \dom(T_2),   
\end{split} 
\end{align}
denote the graph Hilbert spaces associated with the closed operators $T_1$, $T_2$, yields
\begin{equation}
\wti T_1 \in \cB(\cH_{T_1},\cH_2), \quad \wti T_2 \in \cB(\cH_{T_2},\cH_3). 
\end{equation}
Since 
\begin{equation}
\ker(T_j) = \ker\big(\wti T_j\big), \quad \ker(T_j^*) = \ker\big(\big(\wti T_j\big)^*\big), \quad j=1,2,
\end{equation}
the case of unbounded, closed operators is now reduced to that of bounded operators, see, \cite[Remark~1.3.27]{EE18}, \cite[p.~372]{GGK90}, \cite[Corollary~7.6, Lemma~7.7]{Sc02}.

The only subtle detail that does not automatically transfer from the case of unbounded operators to that of bounded operators is the property of $T_2 T_1$ being densely defined in $\cH_1$ in the case of left semi-Fredholm operators. For more details in this context see \cite[p.~176--177]{Sc02}.  \\[1mm]
$(ii)$ For the case of linear relations, see, for instance, Sandovici and de Snoo \cite{SdS09}. 
\hfill $\diamond$
\end{remark}

Theorem \ref{t2.4} can now be used to derive a slight improvement of Glazman's result \cite[\S~2]{Gl53}, \cite[Theorem~21, p.~26]{Gl65}, as follows:

\begin{theorem} \lb{t2.6}
Suppose $T_j$ are densely defined in $\cH$ and closed, with $0 \in \hatt \rho(T_j)$, $j=1,2$. \\[1mm]
$(i)$ Assume in addition that $\defe(T_j,0) < \infty$, $j=1,2$. Then $T_1 T_2$ and $T_2 T_1$ are densely defined and closed, $T_1$, $T_2$, $T_1 T_2$, $T_2 T_1$ are Fredholm, 
\begin{equation}
0 \in \hatt \rho(T_1 T_2) \cap \hatt \rho(T_2 T_1),    \lb{2.33}
\end{equation}
and 
\begin{equation}
\defe(T_1 T_2,0) = \defe(T_1,0) + \defe(T_2,0) = \defe(T_2 T_1,0) \in \bbN_0. 
\end{equation}
$(ii)$ Assume in addition that $T_1 T_2$ is densely defined. Then $T_1 T_2$ is closed, $T_1$, $T_2$, $T_2 T_1$ are left semi-Fredholm, 
\begin{equation}  
0 \in \hatt \rho(T_1 T_2),
\end{equation}
and 
\begin{equation}
\defe(T_1 T_2,0) = \defe(T_1,0) + \defe(T_2,0) \in \bbN_0 \cup \{\infty\}. 
\end{equation}
\end{theorem} 
\begin{proof}
Since $0 \in \hatt \rho(T_j)$, $j=1,2$, one necessarily has \eqref{2.33} as 
\begin{equation}
\|T_2T_1 f\|_{\cH} \geq c_{T_2,0} \|T_1 f\|_{\cH} \geq c_{T_2,0} c_{T_1,0} \|f\|_{\cH}, \quad f \in \dom (T_2T_1) \subseteq \dom(T_1).
\end{equation} 
Analogously,
\begin{equation}
\|T_1T_2 f\|_{\cH} \geq c_{T_1,0} \|T_2 f\|_{\cH} \geq c_{T_1,0} c_{T_2,0} \|f\|_{\cH}, \quad f \in \dom (T_1T_2) \subseteq \dom(T_2).
\end{equation} 
Moreover, one has
\begin{equation}
0 \in \hatt \rho(T_1) \cap \hatt \rho(T_2) \cap \hatt \rho(T_2T_1) \cap \hatt \rho(T_1T_2), 
\end{equation}
and hence,
\begin{equation}
\ker(T_1) = \ker(T_2) = \ker(T_2T_1) = \ker(T_1T_2) = \{0\}.  
\end{equation}
$(i)$ Since by hypothesis $\defe(T_j,0) = \dim(\ker(T_j^*)) < \infty$, this implies that $T_j$, $j=1,2$, are Fredholm and hence $T_2 T_!$ and $T_1 T_2$ are densely defined, closed, and Fredholm, and by Theorem \ref{t2.4}\,$(i)$, 
\begin{align}
& \ind(T_2 T_1) = - \dim(\ker((T_2 T_1)^*))    \no \\
& \quad = - \defe(T_2 T_1,0) = \ind(T_1) + \ind(T_2) = -  \dim(\ker( T_1^*)) -  \dim(\ker((T_2^*))    \no \\
& \quad = - \defe(T_1,0) - \defe(T_2,0),   \lb{2.41} \\
& \ind(T_1 T_2) = - \dim(\ker((T_1 T_2)^*))   \no  \\
& \quad = - \defe(T_1 T_2,0) = \ind(T_1) + \ind(T_2) = -  \dim(\ker( T_1^*)) -  \dim(\ker((T_2^*))   \no \\
& \quad = - \defe(T_1,0) - \defe(T_2,0). 
\end{align}
$(ii)$ Since $\ker(T_1) = \ker(T_2) = \{0\}$, $T_j$, $j=1,2$, are left semi-Fredholm. Together with the hypothesis that $T_2 T_1$ is densely defined, this yields that $T_2 T_!$ is closed and left semi-Fredholm, and hence Theorem \ref{t2.4}\,$(ii)$ applies to yield \eqref{2.41} again.
\end{proof}

\begin{remark} \lb{r2.7} 
Clearly, iterations in Theorems \ref{t2.4} and \ref{t2.6} yield analogous results for any finite product of appropriate closed operators. \hfill $\diamond$
\end{remark} 

We conclude this section with some applications to powers of densely defined, closed, symmetric operators $S$ in $\cH$. We start with even powers:

\begin{theorem} \lb{t2.8}
Let $S$ be a densely defined, closed, symmetric operator in $\cH$ with deficiency indices $n_{\pm}(S)$ given by
\begin{align}
\begin{split}
n_{\pm}(S) &= \defe(S,\mp i) =  \defe(S \pm i I_{\cH},0) = \dim\big((\ran(S \pm i I_{\cH}))^{\bot}\big)    \\
&= \dim(\ker(S^* \mp i I_{\cH})) \in \bbN \cup \{+\infty\}.
\end{split}
\end{align}
$(i)$ Suppose that $n_{\pm}(S) \in \bbN_0$. Then $S^{2k}$, $k \in \bbN$, is densely defined, symmetric, and closed,  
\begin{equation}
0 \in \hatt \rho\big(S^{2k} + I_{\cH}\big), \quad k \in \bbN, 
\end{equation}
and 
\begin{equation}
n_+\big(S^{2k}\big) = n_-\big(S^{2k}\big) = k [n_+(S) + n_-(S)] \in \bbN, \quad k \in \bbN.     \lb{2.44}
\end{equation}
$(ii)$ Suppose that $n_{\pm}(S) \in \bbN_0 \cup \{+\infty\}$ and that $S^{2\ell}$, $\ell \in \{1,\dots,k\}$, are densely defined. Then $S^{2k}$, $k \in \bbN$, is symmetric and closed, 
\begin{equation}
0 \in \hatt \rho\big(S^{2k} + I_{\cH}\big), \quad k \in \bbN, 
\end{equation} 
and 
\begin{equation}
n_+\big(S^{2k}\big) = n_-\big(S^{2k}\big) = k [n_+(S) + n_-(S)] \in \bbN \cup \{+\infty\}, \quad k \in \bbN.     \lb{2.47}
\end{equation}
\end{theorem}
\begin{proof}
Let $k \in \bbN$ throughout this proof. Consider $S^{2k} + I_{\cH}$, and factor it into
\begin{equation}
S^{2k} + I_{\cH} = \prod_{j=1}^{2k} (S - \omega_j), \quad \omega_j = e^{(2j-1)i\pi/(2k)}, \; j \in \{1,2,\dots,2k\}.
\end{equation}
Then $0 \in \hatt \rho\big(S^{2k} + I_{\cH}\big)$ and hence 
\begin{equation}
n_+\big(S^{2k} + I_{\cH}\big) = n_-\big(S^{2k} + I_{\cH}\big), \quad k \in \bbN.     
\end{equation} 

 To prove item $(i)$, one repeatedly applies Theorem \ref{t2.6}\,$(i)$ to obtain 
 \begin{equation}
 n_{\pm}\big(S^{2k} + I_{\cH}\big) = \defe\big(S^{2k} + I_{\cH},0\big) = \sum_{j=1}^{2k} \defe(S - \omega_j,0) = k[n_+(S) + n_-(S)], 
 \end{equation} 
 since by symmetry, $k$ of the $\omega_j$ lie in the open upper complex half-plane $\bbC_+$, contributing the term $k n_+(S)$, and $k$ of the $\omega_j$ lie in the open lower complex half-plane $\bbC_-$, contributing the term $k n_-(S)$.   Since $S^{2k} + I_{\cH}$ and $S^{2k}$ have the same deficiency indices, one obtains \eqref{2.44}. 
 
The proof of item $(ii)$ then follows in the same manner upon repeatedly applying Theorem \ref{t2.6}\,$(ii)$.
\end{proof}

The case $k=1$ in Theorem \ref{t2.8}, assuming $n_{\pm}(S) \in \bbN$, is due to Glazman \cite[Theorem~22, p.~26]{Gl65}.

Next, we treat the case of odd powers of $S$:

\begin{theorem} \lb{t2.9}
Let $S$ be a densely defined, closed, symmetric operator in $\cH$ with deficiency indices $n_{\pm}(S)$ given by
\begin{align}
\begin{split}
n_{\pm}(S) &= \defe(S,\mp i) =  \defe(S \pm i I_{\cH},0) = \dim\big((\ran(S \pm i I_{\cH}))^{\bot}\big)    \\
&= \dim(\ker(S^* \mp i I_{\cH})) \in \bbN_0 \cup \{+\infty\}.
\end{split}
\end{align}
$(i)$ Suppose that $n_{\pm}(S) \in \bbN_0$. Then $S^{2k+1}$, $k \in \bbN$, is densely defined, symmetric, and closed,  
\begin{equation}
0 \in \hatt \rho\big(S^{2k+1} \pm i I_{\cH}\big), \quad k \in \bbN,
\end{equation}
and 
\begin{align}
& n_+\big(S^{2k+1}\big) = k [n_+(S) + n_-(S)] + n_+(S) \in \bbN_0, \quad k \in \bbN.     \lb{2.53} \\
& n_-\big(S^{2k+1}\big) = k [n_+(S) + n_-(S)] + n_-(S) \in \bbN_0, \quad k \in \bbN. 
\end{align}
$(ii)$ Suppose that $n_{\pm}(S) \in \bbN_0 \cup \{+\infty\}$ and that $S^{2\ell+1}$, $\ell \in \{1,\dots,k\}$, are densely defined. Then $S^{2k+1}$, $k \in \bbN$, is symmetric and closed, 
\begin{equation}
0 \in \hatt \rho\big(S^{2k+1} \pm i I_{\cH}\big), \quad k \in \bbN, 
\end{equation} 
and 
\begin{align}
& n_+\big(S^{2k+1}\big) = k [n_+(S) + n_-(S)] + n_+(S) \in \bbN_0 \cup \{+\infty\}, \quad k \in \bbN.     \lb{2.56} \\
& n_-\big(S^{2k+1}\big) = k [n_+(S) + n_-(S)] + n_-(S) \in \bbN_0 \cup \{+\infty\}, \quad k \in \bbN. 
\end{align}
\end{theorem}
\begin{proof}
Let $k \in \bbN$ throughout this proof. Consider $S^{2k+1} \pm i I_{\cH}$, and factor it into
\begin{align}
& S^{2k+1} \pm i I_{\cH} = \prod_{j=1}^{2k+1} (S - \omega_{\pm,j}), \\
\begin{split}
& \omega_{+,j} = e^{3 \pi i/[4k+2]} \omega_{2k+1}^{j-1}, \quad \omega_{-,j} = e^{\pi i/[4k+2]} \omega_{2k+1}^{j-1}, 
\quad j \in \{1,\dots,2k\},     \\
& \omega_{2k+1} = e^{2\pi i/[2k+1]}, \; k \in \bbN. 
\end{split}
\end{align}
Since
\begin{align}
& \big\|\big(S^{2k+1} \pm i I_{\cH}\big)f\big\|^2_{\cH} = \big(|\big(S^{2k+1} \pm i I_{\cH}\big)f, |\big(S^{2k+1} \pm i I_{\cH}\big)f\big)_{\cH}     \no \\
& \quad = \big\|S^{2k+1}f\big\|^2_{\cH} \pm i \big(S^{2k+1}f,f\big)_{\cH} \mp i \big(f,S^{2k+1}f\big)_{cH} + \|f\|^2_{\cH} 
\no \\
& \quad = \big\|S^{2k+1}f\big\|^2_{\cH} \pm i \big(f, (S^*)^{2k+1}f\big)_{\cH} \mp i \big(f,S^{2k+1}f\big)_{cH} + \|f\|^2_{\cH} 
\no \\
& \quad = \big\|S^{2k+1}f\big\|^2_{\cH} \pm i \big(f, S^{2k+1}f\big)_{\cH} \mp i \big(f,S^{2k+1}f\big)_{cH} + \|f\|^2_{\cH} 
\no \\
& \quad = \big\|S^{2k+1}f\big\|^2_{\cH} + \|f\|^2_{\cH} \geq \|f\|^2_{\cH}, \quad f \in \dom\big(S^{2k+1}\big) 
\subseteq \dom\big((S^*)^{2k+1}\big),
\end{align}
one obtains 
\begin{equation}
0 \in \hatt \rho\big(S^{2k+1} \pm i I_{\cH}\big).
\end{equation}  

To prove item $(i)$, one repeatedly applies Theorem \ref{t2.6}\,$(i)$ to obtain 
 \begin{align}
 \begin{split}
 n_{\pm}\big(S^{2k+1}\big) &= \defe\big(S^{2k+1} \pm i I_{\cH},0\big) = \sum_{j=1}^{2k+1} \defe(S - \omega_{\pm,j},0) \\
 &= k[n_+(S) + n_-(S)] + n_{\pm}(S), 
 \end{split} 
 \end{align} 
 since $k+1$ of the $\omega_{+,j}$ lie in the open upper complex half-plane $\bbC_+$, contributing the term $(k+1) n_+(S)$, and $k$ of the $\omega_{+,j}$ lie in the open lower complex half-plane $\bbC_-$, contributing the term $k n_-(S)$.   
Similarly, $k+1$ of the $\omega_{-,j}$ lie in the open lower complex half-plane $\bbC_-$, contributing the term $(k+1) n_-(S)$, and $k$ of the $\omega_{-,j}$ lie in the open upper complex half-plane $\bbC_+$, contributing the term $k n_+(S)$.    
 
The proof of item $(ii)$ then follows in the same manner upon repeatedly applying Theorem \ref{t2.6}\,$(ii)$.
\end{proof}

Theorems \ref{t2.8} and \ref{t2.9} are due to Behncke and Focke \cite{BF77}. Their proof is rather succinct so we decided to provide the arguments based on semi-Fredholm operator techniques in detail. 

In particular, in the special case of equal deficiency indices, $n_+(S) = n_-(S) = n(S)$, then (under the hypotheses in Theorems \ref{t2.8} and \ref{t2.9}) 
\begin{equation}
n_{\pm}\big(S^m\big) = m \, n(S), \quad m \in \bbN,    \lb{2.61}
\end{equation}
also recovering a result of Glazman \cite[p.~27]{Gl65}.

In connection with the denseness hypothesis in Theorems \ref{t2.8}\,$(ii)$ and \ref{t2.9}\,$(ii)$ we recall the following result by Schm\"{u}dgen:   

\begin{theorem} $($\cite[Theorem~1.9]{Sc83}$)$.  \lb{l2.3} ${}$\\
Assume $T$ is densely defined, closed, and symmetric in $\cH$, and suppose that at least one of its deficiency indices is finite. Then $\bigcap_{k \in \bbN} \dom\big(T^k\big)$ is a core for each $T^n$, $n \in \bbN$. In particular, $\bigcap_{k \in \bbN} \dom\big(T^k\big)$ $\big($and hence, $\dom\big(T^k\big)$$\big)$, $k \in \bbN$, is dense in $\cH$.
\end{theorem}

Next, we re-derive the fact that the deficiency indices of a real polynomial $P_m$ of degree $m \in \bbN$ of a given symmetric operator $S$ of the form
\begin{equation}
P_m(S) = a_m S^m+a_{m-1}S^{n-1}+\dots+a_1S+a_0, \quad a_m > 0, \, a_j \in \bbR, \; 0 \leq j \leq m-1,    \lb{2.64} 
\end{equation}
satisfy 
\begin{equation}
n_{\pm}(P_m(S))=n_\pm(S^m). 
\end{equation}
This is a well-known result, see, for instance, \cite[Thm.~1 and p. 126]{BF77}, and can be shown using the functional analytic fact that the lower order terms in $P_m$ are relatively bounded with respect to $S^m$ with relative bound equal to zero (i.e., they are infinitesimally bounded w.r.t. $S^m$). In what follows, we provide a different, complex analytic proof of this fact by counting the numbers of roots of $P_m(t)\pm i\varepsilon$ that lie in $\bbC_+$ and $\bbC_-$ for $\varepsilon>0$ sufficiently small and then invoke Theorem \ref{t2.6}\,$(ii)$.

The following Lemma \ref{l2.11} is known, in fact, a more general version, the Cauchy Index Theorem, was formulated by Hurwitz \cite{Hu95} and is recorded, for instance, in \cite[Thm.~(37,1), p. 169--170]{Ma89} (see also \cite[Sect.~11.3]{RS02}). For convenience of the reader we present the following elementary proof that was kindly communicated to us by Andrei Martinez-Finkelshtein: 

\begin{lemma} \lb{l2.11}
Assume that $P_m(z)$, $z \in \bbC$, is a real-polynomial of degree $m \in \bbN$ with positive highest coefficient $a_m > 0$ in \eqref{2.64}, and with $m$ simple roots. In addition, introduce 
$P_{m, \varepsilon}^\pm (z) = P_m(z)\mp i\varepsilon$, $z \in \bbC$, $\varepsilon \in (0,\infty)$, and let $k \in \bbN$. Then the following items $(i)$--$(iii)$ hold for $0 < \varepsilon$ sufficiently small\,$:$ \\[1mm] 
$(i)$ If $m=2k$, then the polynomials $P_{m,\varepsilon}^\pm$ have exactly $k$ simple roots in $\bbC_+$ and $k$ simple roots in $\bbC_-$. \\[1mm]
$(ii)$ If $m=2k-1$, then the polynomial $P_{m,\varepsilon}^+(t)$ has exactly $k$ simple roots in $\bbC_+$ and $(k-1)$ simple roots in $\bbC_-$. \\[1mm] 
$(iii)$ If $m=2k-1$, then the polynomial $P_{m,\varepsilon}^-$ has exactly $(k-1)$ simple roots in $\bbC_+$ and $k$ simple roots in $\bbC_-$. \\[1mm]
In addition, one obtains the following result global in $\varepsilon > 0$\,$:$ \\[1mm]
$(iv)$ For all $\varepsilon \in (0,\infty)$, the roots of $P_{m,\varepsilon}^\pm$ remain in the open half-plane $\bbC_+$, respectively, $\bbC_-$, they originally entered for $0 < \varepsilon$ sufficiently small; equivalently, they cannot change half-planes as $\varepsilon$ runs through the interval $(0,\infty)$. $($The roots are not necessarily simple, in general.$)$
\end{lemma}
\begin{proof}
Since $\mu\in\bbC$ is a root of $P_{m,\varepsilon}^+$ if and only if $\ol{\mu}$ is a root of $P_{m,\varepsilon}^-$, assertion $(iii)$ follows from assertion $(ii)$, and assertion $(i)$ for $P_{m,\varepsilon}^-$ follows from assertion $(i)$ for $P_{m,\varepsilon}^+$. Thus, it suffices to prove items $(i)$ (for $P_{m,\varepsilon}^+$), $(ii)$, and $(iv)$.

As a special case of the Lagrange inversion theorem, in fact, as a special case of the Lagrange-- B\"urmann formula (see, e.g., \cite[\S~2.4]{He74}, \cite{Wi25}), considering 
\begin{equation}
F(w) = w/\phi(w), \, \text{ $\phi$ analytic near $0$ with $\phi(0) \neq 0$,}
\end{equation}
the equation $F(G(z)) = z$, $z \in \bbC$, $|z|$ sufficiently small, yields 
\begin{equation}
G(z) = \sum_{n \in \bbN} \f{1}{n} \bigg[\f{1}{(n-1)!} \big(\phi(w)^n\big)^{(n-1)}\bigg] z^n = \phi(0) z + \Oh\big(z^2\big). 
\lb{2.67} 
\end{equation}
An application of \eqref{2.67} to $\phi(w) = w/P_m(w+z_0) = (z-z_0)/P_m(z)$ with $w=z-z_0$ and $z_0$ a simple zero of $P_m(\dott)$, then shows that $w/\phi(w) = P_m(z)$ and $\phi(0) = 1/P_m'(z_0) \neq 0$. Thus, 
\begin{equation}
w/\phi(w) = i \varepsilon \, \text{ for $0 < \varepsilon$ sufficiently small,} 
\end{equation}
has the unique solution
\begin{equation}
w(\varepsilon) = z_0(\varepsilon) - z_0 = [1/P_m'(z_0)] i \varepsilon + \Oh\big(\varepsilon^2\big).
\end{equation}
If $z_m < \cdots < z_1$ are the real and simple zeros of $P_m$, then $a_m > 0$ yields  
\begin{equation}
\sgn(P_m'(z_k)) = (-1)^{k-1}, \quad k = 1,\dots,m,
\end{equation}
proving assertions $(i)$ (for $P_{m,\varepsilon}^+$) and $(ii)$. 

To prove assertion $(iv)$ one can argue as follows: As $\varepsilon$ varies through the interval $(0,\infty)$, the simple zero $z_0(\varepsilon)$ moves off the real axis, and if it ever would change its half-plane again, this could only happen under two possible scenarios, both of which will be refuted next: First, by continuity in $\varepsilon$, it would have to cross the real axis at some point $\varepsilon_0 \in (0,\infty)$, however, $P_m(\dott) + i \varepsilon_0$ cannot have any real zeros. Second, zeros could transition to the other open complex half-plane by going through infinity. However, the additive perturbation $i \varepsilon$ cannot imply such a scenario either. Thus, $z_0(\varepsilon)$ necessarily stays in the open complex half-plane it first entered when $0 < \varepsilon$ was sufficiently small. 
\end{proof} 

\begin{corollary} \lb{c2.12}
Let $P_m$ be a real polynomial of degree $m \in \bbN$ with positive highest coefficient $a_m>0$ $($cf.\ \eqref{2.64}$)$ and $S$ a closed and symmetric operator in $\cH$. Then the defeciency indices of $P_m(S)$ are given by 
\begin{equation}
n_{\pm}(P_m(S))=n_\pm(S^m).
\end{equation} 
\end{corollary}
\begin{proof}
If the roots of $P_m$ are not simple, there exists $c\in\bbR$ such that the roots of the polynomial $\wti P_m(z) = P_m(z)+c$, $z \in \bbC$, are simple. Since defect indices are constant in $\bbC_+$ as well as $\bbC_-$, it follows that $n_{\pm}\big(\wti P_m(S)\big)=n_\pm(P_m(S))$. Corollary \ref{c2.12} now follows from Theorem \ref{t2.6}\,$(ii)$ and Lemma \ref{l2.11}.
\end{proof}

\section{Some Applications to Ordinary and \\ Partial Differential Operators} \lb{s3}

In the first part of this section we apply the principal results of Section \ref{s2} to powers of certain minimal Sturm--Liouville operators that originally motivated the writing of this note and also deal with minimal higher even-order ordinary differential operators. In the second part we consider two applications to partial differential operators. 

\subsection{A Bessel-Type Differential Operator and Its Square.}

We start this subsection by re-examining an example that was first analyzed by Chaudhuri and Everitt \cite{CE69}, but present it in a new light as a particular Bessel-type differential expression on $(0,1)$.

Introduce
\begin{equation}
\tau_{2,\alpha} = - \f{d^2}{dx^2} + \f{\alpha^2 - (1/4)}{(1-x)^2}, \quad \alpha \in [1,\infty), \; x \in [0,1),   \lb{3.1} 
\end{equation}
implying
\begin{align}
\begin{split} 
\tau_{4,\alpha} = \tau_{2,\alpha}^2 = \f{d^4}{dx^4} - \f{d}{dx} \f{2 \alpha^2 - (1/2)}{(1-x)^2} \f{d}{dx} 
+ \f{\alpha^4 - (13/2) \alpha^2 + (5/4)^2}{(1-x)^4},&  \\
\alpha \in [1,\infty), \; x \in [0,1).&
\end{split}
\end{align}
The first two quasi-derivatives associated with $\tau_{2,\alpha}$ then are of the type 
\begin{equation}
g^{[0]}(x) = g(x), \; g^{[1]}(x) = g'(x), \quad x \in [0,1), \; g, g' \in AC_{loc}([0,1)), 
\end{equation} 
similarly, the first four quasi-derivatives at $x=0$ associated with $\tau_{4,\alpha}$ then are of the type 
\begin{align}
\begin{split}
& g^{[0]}(x) = g(x), \; g^{[1]}(x) = g'(x), \; g^{[2]}(x) = g''(x),     \\
& g^{[3]}(x) = \big[2 \alpha^2 - (1/2)\big] (1-x)^{-2} g'(x) - g'''(x), 
\quad x \in [0,1), \\ 
& \hspace*{4.9cm} g, g', g'', g''' \in AC_{loc}([0,1)).  
\end{split} 
\end{align} 
Here we used the fact that the first four quasi-derivatives associated with the fourth-order differential expression regular at $x=0$ 
\begin{align} 
\begin{split} 
& \f{d^2}{dx^2} p_0(x) \f{d^2}{dx^2} - \f{d}{dx} p_1(x) \f{d}{dx} + p_2(x), \quad x \in [0,1), \\
& (1/p_0), p_1, p_2 \in L^1_{loc}([0,1)),  
\end{split}
\end{align}
are of the type (cf.\ \cite[\S~17.3]{Na68}) 
\begin{align} 
\begin{split} 
& g^{[0]}(x) = g(x), \; g^{[1]}(x) = g'(x), \; g^{[2]}(x) = p_0(x) g''(x),    \\
& g^{[3]}(x) = p_1(x) g'(x) - \big[g^{[2]}\big]'(x), \quad x \in [0,1), \; g^{[k]} \in AC_{loc}([0,1)), \, 0 \leq k \leq 3.
\end{split} 
\end{align} 

Consequently,
\begin{equation}
g^{[k]} \in AC_{loc}([0,1)) \, \text{ if and only if } \, g^{(k)} \in AC_{loc}([0,1)), \; 0 \leq k \leq 3,
\end{equation}
and 
\begin{equation}
g^{[k]}(0) = 0 \, \text{ if and only if } \, g^{(k)}(0) = 0, \; 0 \leq k \leq 3.
\end{equation}

Next, introducing 
\begin{equation}
u_{\beta}(x) = (1-x)^{\beta}, \quad \beta \in \bbR, \; x \in [0,1),
\end{equation}
one obtains
\begin{equation}
(\tau_{2,\alpha} u_{b_j})(x) = 0, \quad b_j = (1/2) + (-1)^{j+1} \alpha, \; j=1,2.
\end{equation}
Similarly, 
\begin{equation}
(\tau_{4,\alpha} u_{b_j})(x) = \Big(\big(\tau_{2,\alpha}\big)^2 u_{b_j}\Big)(x) = 0, \quad 
b_j = \begin{cases} (1/2) + (-1)^{j+1} \alpha, & j=1,2, \\ 
(5/2) + (-1)^{j+1} \alpha, & j=3,4.
\end{cases} \\ 
\end{equation}

Thus, one  obtains
\begin{align}
\begin{split}
& \tau_{2,\alpha} u_{b_1} = 0 = \tau_{2,\alpha} u_{b_2}, \quad \alpha \in [1,\infty),   \\
& \tau_{2,\alpha} u_{b_3} = - 4(1+\alpha) u_{b_1}, \quad \alpha \in [1,\infty),   \\
& \tau_{2,\alpha} u_{b_4} = - 4(1-\alpha) u_{b_2}, \quad \alpha \in [1,\infty),   \lb{3.13} \\
& u_{b_1},  u_{b_3},  u_{b_4} \in L^2((0,1)), \;  u_{b_2} \notin L^2((0,1)), \quad \alpha \in [1, 3).   
\end{split} 
\end{align}
In particular, $\tau_{2,\alpha}$, $\alpha \in [1,\infty)$, and $\tau_{4,\alpha}$, $\alpha \in [1,\infty)$, are regular at $x=0$, but since $u_{b_2} \notin L^2([0,1))$ for $\alpha \in [1, \infty)$, $\tau_{2,\alpha}$, $\alpha \in [1,\infty)$, is in the limit point case at the singular endpoint $x=1$.

Minimal and maximal operators associated with $\tau_{2,\alpha}$ and $\tau_{4,\alpha}$ are then of the following form (here $\alpha \in [1,\infty)$):
\begin{align}
& T_{2,\alpha,max} f = \tau_{2,\alpha} f,   \\
& f \in \dom(T_{2,\alpha,max}) = \big\{g \in L^2((0,1)) \, \big| \, g, g' \in AC_{loc}([0,1)); \, \tau_{2,\alpha} g \in L^2((0,1))\big\},     \no \\
& T_{2,\alpha,min} f = \tau_{2,\alpha} f,   \\
& f \in \dom(T_{2,\alpha,min}) = \{g \in \dom(T_{2,\alpha,max}) \, | \, g(0) = g'(0) = 0\},   \no  \\
& T_{4,\alpha,max} f = \tau_{4,\alpha} f,   \\
& f \in \dom(T_{4,\alpha,max}) = \big\{g \in L^2((0,1)) \, \big| \, g, g', g'', g''' \in AC_{loc}([0,1));    \no \\ 
& \hspace*{7.3cm} \tau_{4,\alpha} g \in L^2((0,1))\big\},     \no \\
& T_{4,\alpha,min} f = \tau_{4,\alpha} f,   \\
& f \in \dom(T_{4,\alpha,min}) = \{g \in \dom(T_{4,\alpha,max}) \, | \, g(0) = g'(0) = g''(0) = g'''(0) = 0\},   \no
\end{align}
with 
\begin{align} 
\begin{split} 
& T_{2,\alpha,max}^* = T_{2,\alpha,min}, \quad T_{2,\alpha,min}^* = T_{2,\alpha,max},     \\
& T_{4,\alpha,max}^* = T_{4,\alpha,min}, \quad T_{4,\alpha,min}^* = T_{4,\alpha,max}; \quad \alpha \in [1,\infty). 
\end{split} 
\end{align}
Similarly,
\begin{align}
& T_{2,\alpha,max}^2 f = \tau_{4,\alpha} f,   \\
& f \in \dom\big(T_{2,\alpha,max}^2\big) = \big\{g \in L^2((0,1)) \, \big| \, g, g', g'', g''' \in AC_{loc}([0,1));    \no \\ 
& \hspace*{6.5cm} \tau_{2,\alpha} g, \tau_{4,\alpha} g \in L^2((0,1))\big\}     \no \\
& \hspace*{2.9cm} = \dom(T_{2,\alpha,max}) \cap \dom(T_{4,\alpha,max}), \quad \alpha \in [1,\infty),    \\
& \hspace*{2.9cm} \subsetneqq \dom(T_{2,\alpha,max}), \,  \dom(T_{4,\alpha,max}), \quad \alpha \in [1,3),      \\
& T_{2,\alpha,min}^2 f = \tau_{4,\alpha} f,   \\
& f \in \dom\big(T_{2,\alpha,min}^2\big) = \big\{g \in \dom\big(T_{2,\alpha,max}^2\big)  \, \big| \, g(0)=g'(0)=g''(0)=g'''(0)=0\}     \no \\
& \hspace*{2.9cm} = \dom(T_{2,\alpha,min}) \cap \dom(T_{4,\alpha,min}), \quad \alpha \in [1,\infty),      \\
& \hspace*{2.9cm} \subsetneqq \dom(T_{4,\alpha,min}), \, \dom(T_{2,\alpha,min}), \quad \alpha \in [1,3).   \lb{3.24} 
\end{align}
Indeed, assuming $\alpha \in [1,3)$ throughout \eqref{3.25}--\eqref{3.34}, 
\begin{equation}
u_{b_4} \in \dom(T_{4,\alpha,max}), \quad u_{b_4} \notin \dom(T_{2,\alpha,max}),     \lb{3.25}
\end{equation}
and similarly, introducing 
\begin{equation}
(\wti u_{b_4})(x) = \begin{cases} 0, & x \in [0,1/4], \\
u_{b_4}(x), & x \in [1/2,1),
\end{cases} \quad   \wti u_{b_4} \in C^{\infty}((0,1)), 
\end{equation}
then 
\begin{equation}
\wti u_{b_4} \in \dom(T_{4,\alpha,min}), \quad \wti u_{b_4} \notin \dom(T_{2,\alpha,min}),
\end{equation}
In addition, recalling $u_{\beta}(x) = (1-x)^{\beta}$, $x \in [0,1)$, one computes 
\begin{align}
& (\tau_{2,\alpha} u_{\beta})(x) = \big[- \beta(\beta-1) + \alpha^2 - (1/4)\big] (1-x)^{\beta - 2},    \\
& (\tau_{4,\alpha} u_{\beta})(x) = \big[\alpha^2 -(1/4) - \beta(\beta-1)\big]\big[\alpha^2 - (1/4) - (\beta-2)(\beta-3)\big] (1-x)^{\beta-4}   \no \\
& \hspace*{1.75cm} = \big[\beta(\beta-1)(\beta-2)(\beta-3) - \big[2 \alpha^2 - (1/2)\big] \beta(\beta-3)   \no \\
& \hspace*{2.3cm} + \alpha^4 - (13/2) \alpha^2 + (5/4)^2\big] (1-x)^{\beta - 4}; \quad x \in [0,1),
\end{align} 
and assuming $\beta = 3$ as well as  
\begin{equation}
\alpha^2 - 6 - (1/4) \neq 0, \quad \alpha^2 - (1/4) \neq 0, 
\end{equation}
one concludes that
\begin{equation}
u_3, \tau_{2,\alpha} u_3 \in L^2([0)), \quad \tau_{4,\alpha} u_3 \notin L^2([0,1)),
\end{equation}
and hence,
\begin{equation}
u_3 \in \dom(T_{2,\alpha,max}), \quad u_3 \notin \dom(T_{4,\alpha,max}). 
\end{equation}
Finally, introducing 
\begin{equation}
(\wti u_{3})(x) = \begin{cases} 0, & x \in [0,1/4], \\
u_3(x), & x \in [1/2,1),
\end{cases} \quad   \wti u_3 \in C^{\infty}((0,1)), 
\end{equation}
one obtains 
\begin{equation}
\wti u_3 \in \dom(T_{2,\alpha,min}), \quad \wti u_3 \notin \dom(T_{4,\alpha,min}).     \lb{3.34}
\end{equation}

Next, we recall the following well-known fact: 
\begin{proposition} $($See, e.g., \cite{AD12}, \cite{Gu69}, \cite{Ho69}, \cite{Sc70}, \cite{vCG70}.$)$ ${}$ \lb{p3.1} \\
Suppose that $T_j$, $j=1,2$, are densely defined in the complex, separable Hilbert space $\cH$, $T_1$ is closed in $\cH$, $\dim(\ker(T_1^*)) < \infty$, then $T_2 T_1$ is densely defined and 
\begin{equation}
T_1^* T_2^* = (T_2 T_1)^*.     \lb{3.35}
\end{equation}
\end{proposition}
We also refer, for instance, to \cite{FL77}, \cite[App.~B]{GGHT12}, \cite{Gu11}, \cite{Ho68}, \cite{vC72}, and the references therein for more facts in connection with \eqref{3.35}.

Clearly, this fact applies to $T_1=T_2 = T_{2,\alpha,max}$, $\alpha \in [1,\infty)$, and to $T_1=T_2 = T_{2,\alpha,min}$, $\alpha \in [1,\infty)$, implying 
\begin{align} 
\begin{split} 
& T_{2,\alpha,max}^2 = \big(T_{2,\alpha,min}^*\big)^2 = \big(T_{2,\alpha,min}^2\big)^*,    \\
& T_{2,\alpha,min}^2 = \big(T_{2,\alpha,max}^*\big)^2 = \big(T_{2,\alpha,max}^2\big)^*; \quad \alpha \in [1,\infty). 
\end{split} 
\end{align}

At this point we can make the connection with indices and defect numbers as discussed in Section \ref{s2}. Assuming 
$\alpha \in [1,3)$, one obtains from \eqref{3.13}--\eqref{3.24},  
\begin{align}
& \ker(T_{2,\alpha,min}) = \{0\},   \\
& \ker(T_{2,\alpha,min}^*) = \{u_{b_1}\}, \quad \dim( \ker(T_{2,\alpha,min}^*)) = 1,    \\
& \ind(T_{2,\alpha,min}) = \dim(\ker(T_{2,\alpha,min})) - \dim( \ker(T_{2,\alpha,min}^*)) = - 1,   \\
& \defe(T_{2,\alpha,min},0) = 1; \quad \alpha \in [1,3), 
\end{align}
and 
\begin{align}
& \ker(T_{4,\alpha,min}) = \{0\},   \\
& \ker(T_{4,\alpha,min}^*) = \{u_{b_1}, u_{b_3}, u_{b_4}\}, \quad \dim( \ker(T_{4,\alpha,min}^*)) = 3,    \\
& \ind(T_{4,\alpha,min}) = \dim(\ker(T_{4,\alpha,min})) - \dim( \ker(T_{4,\alpha,min}^*)) = - 3,   \\
& \defe(T_{4,\alpha,min},0) = 3; \quad \alpha \in [1,3). 
\end{align}

Finally,
\begin{align}
& \ker\big(T_{2,\alpha,min}^2)\big) = \{0\},   \\
& \ker\big(\big(T_{2,\alpha,min}^2\big)^*\big) = \{u_{b_1}, u_{b_3}\}, \quad \dim( \ker(T_{2,\alpha,min}^*)) = 2,    \\
& \ind\big(T_{2,\alpha,min}^2\big) = \dim\big(\ker\big(T_{2,\alpha,min}^2\big)\big) - \dim\big(\ker\big(\big(T_{2,\alpha,min}^2\big)^*\big)\big) = - 2,   \\
& \defe\big(T_{2,\alpha,min}^2,0\big) = 2 = 2 \defe(T_{2,\alpha,min},0); \quad \alpha \in [1,3), 
\end{align}
in agreement with \eqref{2.44}.

\begin{remark} \lb{r3.2}
The original example studied by Chaudhuri and Everitt \cite{CE69} was associated with the differential expression
\begin{equation}
\tau_{2,CE} = - \f{d}{dx} \f{1}{6} (x+1)^4 \f{d}{dx} + (x+1)^2, \quad x \in [0,\infty),    \lb{3.49}
\end{equation}
however, the Liouville--Green transform, see \cite[Sect.~3.5]{GNZ24} for details, and some scaling in $x$, reduce  \eqref{3.49} to \eqref{3.1} with $\alpha = \sqrt{33}/2 \in (1,3)$ as follows: Consider the change of variables
\begin{align}
& t(x) = \int_0^x dx' 6^{1/2} (x' +1)^{-2} = 6^{1/2} \big[1 - (x+1)^{-1}\big], \; x \in [0,\infty), \\     
& t(0) = 0, \; t(\infty) = 6^{1/2},   \\
& \wti u(t) = 6^{-1/4} (x+1) u(x), \quad t \in \big[0,6^{1/2}\big), 
\end{align} 
transforms 
\begin{equation}
(\tau_{2,CE} u)(x) = z u(x),  \quad z \in \bbC, \; x \in [0,\infty),
\end{equation}
into 
\begin{equation}
- \ddot {\wti u} (t) + 8 \big[6^{1/2} -t\big]^{-2} \wti u(t) = z \wti u(t), \quad z \in \bbC, \; t \in \big[0,6^{1/2}\big).    \lb{3.53} 
\end{equation}
The scaling transformation $t \mapsto 6^{1/2} t$ then reduces \eqref{3.53} to 
\begin{equation}
(\tau_{2, \sqrt{33}/2} u)(s) = 6 z u(s),  \quad z \in \bbC, \; s \in [0,1) 
\end{equation} 
(cf.\ \eqref{3.1}). In other words, $\tau_{2,\alpha}$, $\alpha \in [1,3)$ represents a one-parameter extension of the original example by Chaudhuri and Everitt \cite{CE69}.
\hfill $\diamond$
\end{remark}

\subsection{Further illustrations of Theorems \ref{t2.8} and \ref{t2.9} in the ODE context}
We continue illustrating the abstract approach of Section \ref{s2} in the concrete case of certain ordinary differential operators generated by limit circle differential expressions.

Letting $n \in \bbN$, and recalling the quasi-derivatives
\begin{align} 
\begin{split} 
& u^{[0]} = u, \; u^{[k]} = u^{(k)}, \; 1 \leq k \leq (n-1), \; u^{[n]} = p_0 u^{(n)},   \lb{3.55} \\
& u^{[n+k]} = p_k u^{[n-k]} - \big(u^{[n+k-1]}\big)', \; 1 \leq k \leq n,
\end{split}
\end{align}
in particular,
\begin{equation}
\tau_{2n} u = u^{[2n]} = \sum_{k=0}^n (-1)^{n-k} \big(p_k u^{(n-k)}\big)^{(n-k)},    \lb{3.56}
\end{equation}
one assumes
\begin{align}
\begin{split}
& 1/p_0, p_1,\dots,p_n \, \text{ are (Lebesgue) measurable in $(a,b) \subseteq \bbR$,}  \\
& 1/p_0, p_1,\dots,p_n \in L^1_{loc}((a,b)).    \lb{3.57}
\end{split} 
\end{align}

\begin{proposition} \lb{p3.3}
Suppose $T_{2n,min}$ is the minimal operator in $L^2((a,b))$ generated by the $2n$-th order symmetric differential expression $\tau_{2n}$ with a.e.~real-valued coefficients $p_0,\dots,p_n$ on $(a,b) \subseteq \bbR$, $n \in \bbN$, satisfying \eqref{3.55}--\eqref{3.57}. In addition, suppose that  $\tau_{2n}$ is in the limit circle case at $a$ and $b$, that is, for some $($and hence for all\,$)$ $z \in \bbC$, $\tau_{2n} y = z y$ has $2n$ linearly independent solutions $y_j$ satisfying $y_j \in L^2((a,b))$, $1 \leq j \leq 2n$. Then, $T_{2n,min}$ has deficiency indices $(2n,2n)$ and $(T_{2n,min})^m$ has deficiency indices $(2mn,2mn)$, $m \in \bbN$. 
\end{proposition}

This is a consequence of \eqref{2.61} and should be compared with \cite[Thm.~V.4.1]{KRZ77}. 

An elementary  concrete example illustrating Proposition \ref{p3.3} (which originally motivated writing this note)   is provided by the Legendre operator on $(-1,1)$:

\begin{example} \lb{e3.4} $(${\rm The Legendre operator in $L^2((-1,1))$}$)$.  ${}$ \\
Consider the Legendre differential expression
\begin{equation}
\tau_{Leg} = - \f{d}{dx} \big(1 - x^2\big) \f{d}{dx}, \quad x \in (-1,1).  
\end{equation}
One verifies that  
\begin{equation}
u_1 (0,x) = 1, \quad u_2 (0,x) = 2^{-1} \ln((1-x)/(1+x)), \quad x \in (-1,1),  
\end{equation}
satisfy
\begin{equation}
\tau_{Leg} u_j = 0, \quad u_j \in L^2((-1,1)), \; j=1,2. 
\end{equation}
Thus, $\tau_{Leg}$ in the limit circle case $($and singular\,$)$ at both endpoints $\pm 1$ $($see, e.g., \cite[Sect.~5.3]{GNZ24}$)$. Consequently, the associated minimal operator $T_{Leg, min}$ associated with $\tau_{Leg}$ in $L^2((-1,1))$ has deficiency indices $(2,2)$. Explicitly, $T_{Leg, min}$ and its adjoint, the maximal operator associated with $\tau_{Leg}$, are given by 
\begin{align}
& T_{Leg, max} f = \tau_{Leg} f,    \no \\
& f \in \dom(T_{Leg, max}) = \big\{g \in L^2((-1,1)) \, \big| \, g, g^{[1]} \in AC_{loc}((-1,1));     \\ 
& \hspace*{6.7cm} \tau_{Leg} g \in L^2((-1,1))\big\},     \no \\
& T_{Leg, min} f = \tau_{Leg} f,    \no \\
& f \in \dom(T_{Leg, min}) = \big\{g \in \dom(T_{Leg, max})  \, \big| \,  \wti g(\pm 1) = \wti g^{\, \prime}(\pm 1) = 0 \big\},    \no \\
& T_{Leg, min}^* = T_{Leg, max}, \quad T_{Leg, max}^* = T_{Leg, min},
\end{align}
where 
\begin{align}
\begin{split} 
\wti g(\pm 1) &= - W(u_1(0, \dott), g)(\pm 1)    \lb{23.7.3.50} \\
&= - (p g')(\pm 1) 
= \lim_{x \to \pm 1} g(x)\big/\big[2^{-1} \ln((1-x)/(1+x))\big] ,   
\end{split} \\ 
\begin{split} 
\wti g^{\, \prime} (\pm 1) &= W(u_2(0, \dott), g)(\pm 1)  \\
&= \lim_{x \to \pm 1} \big[g(x) - \wti g(\pm 1) 2^{-1} \ln((1-x)/(1+x))\big]. \lb{23.7.3.6.16A}
\end{split} 
\end{align}
are the generalized boundary values adapted to $\tau_{Leg}$ 
$($see, \cite{ELM02}, \cite{GLN20}, \cite[Example~13.7.3]{GNZ24}$)$. 

Here we employed the fact that 
\begin{equation} 
g^{[1]}(x) = \big(1-x^2\big) g'(x), \quad x \in (-1,1), 
\end{equation} 
represents the first quasi-derivative of $g \in \dom(T_{Leg, min}^*) = \dom(T_{Leg, max})$, and 
\begin{equation} 
W(h,k)(x) = h(x) k^{[1]}(x) - h^{[1]}(x) k(x), \quad x \in (-1,1), 
\end{equation}
denotes the Wronskian of $h, k \in \dom(T_{Leg, max})$ in connection with the Legendre differential expression $\tau_{Leg}$. 

By Theorem \ref{t2.6}, the deficiency indices of $(T_{Leg, min})^m$ are given by 
\begin{equation}
n_{\pm} \big((T_{Leg,min})^m\big) = 2m, \quad m \in \bbN,
\end{equation} 
where the form of $(T_{Leg,\min })^{m}$
is explicitly given by 
\begin{equation}
(T_{Leg,\min })^{m}f=\tau _{Leg}^{m}f=\sum_{j=1}^{m}(-1)^{j}PS_{m}^{(j)}%
\frac{d^{j}}{dx^{j}}(1-x^{2})^{j}\frac{d^{j}}{dx^{j}} f,
\end{equation} 
for $f \in \rm{dom}$$(T_{Leg,min})^m$, and $PS_{m}^{(j)}$ is the Legendre--Stirling number defined by%
\begin{equation}
PS_{m}^{(j)}=\sum_{k=1}^{j}(-1)^{k+j}\frac{(2k+1)(k^{2}+k)^{m}}{(k+j+1)!(j-k)!}.
\end{equation} 
For further information, see \cite{ELW02}. These Legendre--Stirling numbers are
generalizations of the classical Stirling numbers $S_{m}^{(j)}$ which are connected to the Laguerre expression which we discuss in our next example. \hfill $\ddagger$
\end{example}

\begin{example} \lb{e3.5} $(${\rm The Laguerre, Hermite, and Jacobi operators}$)$. \\
In our last example, as an application of Theorem \ref{t2.6}, we briefly discuss the deficiency indices of integral powers of the three classical second-order expressions of Laguerre, Hermite, and Jacobi. In each case, when $m \in \bbN$, the $m^{th}$ integral power of each of these expressions is explicitly known in Lagrangian symmetric form. We indicate what the deficiency indices are for the corresponding minimal operators generated by each these powers. \\[1mm] 
$(i)$ The Laguerre expression in $L^2\big((0,\infty);x^\alpha e^{-x}dx\big)$$:$ \\[1mm]
For $\alpha \in (-1,\infty)$, the Laguerre differential expression is given by
\begin{equation}
\tau_{\alpha,Lag} = -x^{-\alpha}e^{x} \f{d}{dx}x^{\alpha +1}e^{-x} \f{d}{dx}, \quad x \in (0,\infty).     
\end{equation}
For this expression, $\tau_{\alpha,Lag}$ is in the limit circle case at the singular endpoint $x=0$ in $L^2\big((0,\infty);x^{\alpha}e^{-x}dx\big)$ when $\alpha \in (-1,1)$ and in the limit point case when $\alpha \in (1,\infty)$. At $x=\infty$, $\tau_{\alpha,Lag}$ is in the limit point case for all $\alpha \in (-1,\infty)$. Consequently, the minimal operator $T_{\alpha,Lag, min}$ has deficiency indices given by 
\begin{equation}
n_{\pm}(T_{\alpha,Lag, min}) = \begin{cases}
1, & \alpha \in (-1,1), \\
0, & \alpha \in [1,\infty).
\end{cases}
\end{equation}
Details can be found in \cite{GLN20}, \cite[Example 13.7.4.]{GNZ24}. By Theorem \ref{t2.6}, the deficiency indices of $(T_{\alpha,Lag, min})^m$ are give by 
\begin{equation}
n_{\pm} \big((T_{\alpha,Lag,min})^m\big) = \begin{cases}
m, & \alpha \in (-1,1), \\
0, & \alpha \in [1,\infty),
\end{cases} \quad m \in \bbN.
\end{equation}
The form of the minimal operator $(T_{\alpha,Lag, min})^m$ is given by
\begin{equation}
(T_{\alpha,Lag,\min })^{m}f=x^{-\alpha}e^{x}\tau _{Lag}^{m}f=x^{-\alpha}e^{x}\sum_{j=1}^{m}(-1)^{j}S_{m}^{(j)}%
\frac{d^{j}}{dx^{j}}x^{\alpha +1}e^{-x}\frac{d^{j}}{dx^{j}}f,
\end{equation}
for $f \in \dom\big((T_{\alpha,Lag,min})^m\big)$, where $S_m^{(j)}$ is the classical Stirling number of the second kind defined by
\begin{equation}
S_{m}^{(j)} = \sum_{i=0}^{j}\frac{(-1)^{i+j}}{j!}\binom{j}{i}i^{m}.
\end{equation}
 For more details, see \cite[Section 12]{LW02}.\\[1mm] 
$(ii)$ The Hermite expression in $L^2\big((-\infty,\infty );e^{-x^2}dx\big)$$:$ \\[1mm] 
The Hermite differential expression is given by
\begin{equation}
\tau_{Her} = -e^{x^2} \f{d}{dx}e^{-x^2} \f{d}{dx}, \quad x \in (-\infty,\infty).  
\end{equation}
For this expression, $\tau_{Her}$ is in the limit point case at both singular endpoints $x=\pm \infty$ in $L^2\big((-\infty,\infty);e^{-x^2}dx\big)$ . Hence the minimal operator $T_{Her, min}$ has deficiency indices $(0,0)$ and thus coincides with the corresponding maximal operator $T_{Her, max}$; in particular, the operator $T_{Her, min} = T_{Her, max}$ is self-adjoint in $L^2\big((-\infty,\infty );e^{-x^2}dx\big)$. Consequently, by Theorem \ref{t2.6}, the deficiency indices of $(T_{Her, min})^m$ are also given by 
\begin{equation}
n_{\pm}\big((T_{Her, min})^m\big) = 0, \quad m \in \bbN.
\end{equation}
In this case, the form of the minimal operator $(T_{Her, min})^m$ is given by
\begin{equation}
(T_{Her,min })^{m}f=e^{x^2}\tau_{Her}^{m}f=e^{x^2}\sum_{j=1}^{m}(-1)^{j}S_{m}^{(j)}2^{n-j}%
\frac{d^{j}}{dx^{j}}e^{-x^2}\frac{d^{j}}{dx^{j}}f,
\end{equation}
for $f \in \rm{dom}$$(T_{Her,min})^m$,  where $S_m^{(j)}$ is the classical Stirling number of the second kind as defined above. For full details, see  \cite{ELW00}. \\[1mm]
$(iii)$ The Jacobi expression in $L^2\big((-1,1);(1-x)^{\alpha}(1+x)^{\beta}dx\big)$$:$ \\[1mm] 
For $\alpha, \beta \in (-1,\infty)$, the Jacobi differential expression is defined by
\begin{equation}
\tau_{\alpha,\beta,Jac} = -(1-x)^{-\alpha}(1+x)^{-\beta} \f{d}{dx}(1-x)^{\alpha+1}(1+x)^{\beta+1} \f{d}{dx}, \quad x \in (-1,1).  
\end{equation}
Of course, the Legendre expression is a special case of the Jacobi expression when $\alpha = \beta =0$. We refer the reader to \cite{EKLWY07}, \cite{GLPS24}, and \cite[Example~13.5.4]{GNZ24} for full details of the analytic study of $\tau_{\alpha,\beta,Jac}$. The singular endpoints $x=\pm 1$ satisfy the following limit point/limit circle criteria:
\begin{enumerate}
\item[$(a)$] the endpoint $x=1$ is limit point if $\alpha \in [1,\infty)$; if
$\alpha \in (-1,0)$, $x=1$ is regular, and if $\alpha \in [0,1)$, $x=1$ is limit circle,
nonoscillatory.

\item[$(b)$] the endpoint $x=-1$ is limit point if $\beta \in [1,\infty)$; if
$\beta \in (-1,0)$, $x=-1$ is regular, and if $\beta \in [0,1)$, $x=-1$ is limit circle, nonoscillary.
\end{enumerate}

From this, we see that the deficiency indices of the minimal operator $T_{\alpha,\beta,Jac, min}$
are given by 
\begin{equation}
n_{\pm} (T_{\alpha,\beta,Jac,min})  = \begin{cases} 
0, & \alpha,\beta \in [1,\infty),  \\
1, & \text{$\alpha \in [1,\infty)$ and $\beta\in(-1,1)$} \\
& \text{or $\beta \in [1,\infty)$ and $\alpha\in(-1,1)$,}  \\
2, & \alpha,\beta\in(-1,1).
\end{cases}
\end{equation}
Consequently, by Theorem \ref{t2.6}, the deficiency indices of $(T_{\alpha,\beta,Jac,min})^m$ are given by 
\begin{equation}
n_{\pm} \big((T_{\alpha,\beta,Jac,min})^m\big)  = \begin{cases} 
0, & \alpha,\beta \in [1,\infty),  \\
m, & \text{$\alpha \in [1,\infty)$ and $\beta\in(-1,1)$} \\
& \text{or $\beta \in [1,\infty)$ and $\alpha\in(-1,1)$,}  \\
2m, & \alpha,\beta\in(-1,1),
\end{cases} \quad m \in \bbN. 
\end{equation}

Moreover, the form of the minimal operator $(T_{\alpha,\beta,Jac, min})^m = (1-x)^{-\alpha}(1+x)^{-\beta}\tau_{Jac}^m$ is explicitly given by 
\begin{align} 
& (T_{\alpha,\beta,Jac,min})^{m}f  \\
& \quad =(1-x)^{-\alpha }(1+x)^{-\beta }\sum_{j=1}^{m}(-1)^{j}P^{(\alpha,\beta)}S_{m}^{(j)}\frac{%
d^{j}}{dx^{j}}(1-x)^{\alpha +j}(1+x)^{\beta +j}\frac{d^{j}}{dx^{j}}f,   \no 
\end{align} 
for $f \in \rm {dom}$$(T_{\alpha,\beta,Jac,min})^m$,  where $P^{(\alpha,\beta)}S_m^{(j)}$ is the Jacobi-Stirling number defined by
\begin{equation}
P^{(\alpha,\beta)}S_{n}^{(j)} = \sum_{r=1}^{j}(-1)^{r+j}\frac{\Gamma
(\alpha+\beta+r+2)(\alpha+\beta+2r+1)(r^{2}+(\alpha+\beta)r+r)^{n-1} 
}{(r-1)!(j-r)!\Gamma(\alpha+\beta+j+r+2)}.
\end{equation}
\hfill $\ddagger$
\end{example}

Next, we turn to some illustrations of Theorems \ref{t2.8} and \ref{t2.9} in the PDE context.

\subsection{Homogeneous perturbations of the Laplacian on $\bbR^n$.}  \lb{s3.3}
In this subsection we transition to some PDE applications in connection with (minus) the Laplacian $- \Delta_n$ on $\bbR^n$ and some of its (homogeneous) perturbations of the form $c |x|^{-2}$, $x \in \bbR^n\backslash \{0\}$, $n \in \bbN$, $n \geq 2$, and some $c \in \bbR$. This requires some preparations to which we turn next.

Consider the Bessel operator on $(0,\infty)$ generated by the differential expression 
 \begin{equation}
 \tau_{\gamma} = - \f{d^2}{dx^2} + \f{\gamma^2 - (1/4)}{x^2}, \quad \gamma \in [0,\infty), \; x \in (0,\infty).
 \end{equation}
The associated maximal and minimal operators $T_{\gamma,max}$ and $T_{\gamma,min}$ in $L^2((0,\infty); dx)$ associated with $\tau_{\gamma}$ are then given by (see, e,g., \cite[Ch.~13]{GNZ24})
\begin{align}
& T_{\gamma,max} f = \tau_{\gamma} f, \quad \gamma \in [0,\infty),   \no \\
& f \in \dom(T_{\gamma,max}) = \big\{g \in L^2((0,\infty); dx) \, \big| \, g, g' \in AC_{loc}((0,\infty)) \\  
& \hspace*{6.3cm} \tau_{\gamma} g \in L^2((0,\infty); dx) \big\}    \no \\
\begin{split} 
& T_{\gamma,min} f = \tau_{\gamma} f, \quad \gamma \in [0,1),      \\ 
& f \in \dom(T_{\gamma,min}) = \big\{g \in \dom(T_{\gamma,max}) \, \big| \, 
\wti g(0) = \wti g^{\, \prime}(0) =0 \big\},      
\end{split}   \\
&  T_{\gamma,min} = T_{\gamma,max}, \quad \gamma \in [1,\infty),  
\end{align}
where the generalized boundary values $\wti g(0)$ and $\wti g^{\, \prime}(0)$ are of the form  
\begin{align}
\wti g(0) &= \begin{cases} \lim_{x \downarrow 0} g(x) \big/ \big[x^{1/2} \ln(1/x)\big], & \gamma = 0, \\
\lim_{x \downarrow 0} g(x) \big/ \big[(2\gamma)^{-1} x^{(1/2)-\gamma}\big], & \gamma \in (0,1),
\end{cases}    \\
\wti g^{\,\prime}(0) &= \begin{cases} \lim_{x \downarrow 0} \big[g(x) - \wti g(0) x^{1/2} \ln(1/x)\big] \big/ x^{1/2}, 
& \gamma = 0, \\
\lim_{x \downarrow 0} \big[g(x) - \wti g(0) (2\gamma)^{-1} x^{(1/2)-\gamma}\big] \big/ x^{(1/2)+\gamma}, 
& \gamma \in (0,1).  
\end{cases}
\end{align}

Then
\begin{align}
& T_{\gamma,max}^* = T_{\gamma,min}, \quad T_{\gamma,min}^* = T_{\gamma,max}, 
\quad \gamma \in [0,1), \\
& T_{\gamma,max} = T_{\gamma,max}^* = T_{\gamma,min}^* = T_{\gamma,min}, \quad \gamma \in [1,\infty). 
\end{align}
In particular, $\tau_{\gamma}$ is in the limit circle case at $x=0$ for $\gamma \in [0,1)$ and in the limit point case at $x=0$ for $\gamma \in [1,\infty)$; in addition, $\tau_{\gamma}$ is in the limit point case at $x=\infty$ for $\gamma \in [0,\infty)$. Consequently, taking into account that $\big[\gamma^2 - (1/4)\big] x^{-2}$, $x \in (0,\infty)$, is real-valued, one obtains (cf.\ \eqref{2.61})
\begin{equation}
n_{\pm} (T_{\gamma,min}) = \begin{cases} 1, & \gamma \in [0,1), \\
0, & \gamma \in [1,\infty), 
\end{cases} \quad  
n_{\pm} \big((T_{\gamma,min})^m\big) = \begin{cases} m, & \gamma \in [0,1), \\
0, & \gamma \in [1,\infty),
\end{cases} \quad m \in \bbN. 
\end{equation}

In addition we introduce the following family of Bessel differential expressions 
\begin{align}
\begin{split} 
& \tau_{n,\ell,L,\alpha} = - \f{d^2}{dr^2} + \f{\ell(\ell+n-2) - L(L+n-2)}{r^2},     \\
& n \in \bbN, \, n \geq 2, \; \ell, L \in \bbN_0, \; \alpha \in [-1/4,3/4), \; r \in (0,\infty),    
\end{split} 
\end{align}
such that once again 
\begin{align} 
& n_{\pm} (T_{n,\ell,L,\alpha,min}) = \begin{cases} 1, & 0 \leq \ell \leq L,  \\
0, & \ell \in [L+1,\infty),
\end{cases}    \\ 
& n_{\pm} \big((T_{n,\ell,L,\alpha,min})^m\big) = \begin{cases} m, & 0 \leq \ell \leq L,  \\
0, & \ell \in [L+1,\infty),
\end{cases} \quad m \in \bbN,    \lb{3.90} \\
& n \in \bbN, \, n \geq 2, \; \ell, L \in \bbN_0, \; \alpha \in [-1/4,3/4).    \no
\end{align}
Here, in obvious notation, $T_{n,\ell,L,\alpha,min}$ denotes the minimal operator associated with $\tau_{n,\ell,L,\alpha}$ in $L^2((0,\infty); dr)$, for the parameter ranges 
\begin{equation} 
n \in \bbN, \, n \geq 2, \; \ell, L \in \bbN_0, \; \alpha \in [-1/4,3/4), 
\end{equation} 
which we keep throughout the rest of Subsection \ref{s3.3}. 

Next, we pivot to the PDE context and consider the pre-minimal operator
\begin{equation}
\bigg(- \Delta_n - \f{[(n-1)(n-3)/4] + L(L+n-2)}{|x|^2}\bigg)\bigg|_{C_0^{\infty}(\bbR^n \backslash \{0\})} \lb{3.92} 
\end{equation}
in $L^2(\bbR^n;d^nx)$. Here $-\Delta_n$ represents (minus) the Laplacian differential expression in $\bbR^n$, which, in spherical coordinates $(r,\omega)$, $r \in (0,\infty)$, $\omega \in \bbS^{n-1}$,  amounts to 
\begin{equation}
- \Delta_n = \bigg(- \f{\partial^2}{\partial r^2} - \f{n-1}{r} \f{\partial}{\partial r} - \f{1}{r^2} \Delta_B\bigg), 
\end{equation} 
where $\Delta_B$ abbreviates the Laplace-Beltrami differential expression on the $(n-1)$-dimensional unit sphere $\bbS^{n-1}$ in $\bbR^n$ (cf., e.g., \cite[Sect.~15.5]{GNZ24} for details). 

As a final ingredient we introduce the unitary map
\begin{equation}
U_n \colon \begin{cases} L^2((0,\infty); dr) \to L^2\big((0,\infty); r^{n-1} dr\big),    \\
f \mapsto M_{(\dott)^{-(n-1)/2}} f,
\end{cases}
\end{equation}
where $M_{\psi}$ denotes the maximally defined operator of multiplication in the space $L^2((0,\infty);dr)$ by the (Lebesgue) measurable function $\psi$, that is, 
\begin{equation}
(M_{\psi} f)(r) = \psi(r) f(r) \, \text{ for a.e.~$r \in (0,\infty)$}, \; f \in L^2((0,\infty);dr).  
\end{equation}
One then infers that (see, e.g., \cite[p.~160--161]{RS75})
\begin{align}
T_{n,L,\alpha,min} &= \ol{\bigg(- \Delta_n - \f{[(n-1)(n-3)/4] + L(L+n-2)}{|x|^2}\bigg)\bigg|_{C_0^{\infty}(\bbR^n \backslash \{0\})}}  \\
& = \bigoplus_{\ell \in \bbN_0} U_n T_{n,\ell,L,\alpha,min} U_n^{-1}, 
\end{align}
where, again in obvious notation, $T_{n,L,\alpha,min}$ is the minimal operator associated with the partial differential expression employed in \eqref{3.92}. 

Exploiting \eqref{3.90} (and noting $\ell \in \bbN_0 = \{0\} \cup \bbN$) one infers that 
\begin{equation}
n_{\pm} (T_{n,L,\alpha,min}) = n_{\pm} \Bigg(\bigoplus_{\ell \in \bbN_0} T_{n,\ell,L,\alpha,min} \Bigg) = L+1,
\end{equation}
implying 
\begin{equation}
n_{\pm} \big((T_{n,L,\alpha,min})^m\big) = m(L+1), \quad m \in \bbN, 
\end{equation}
again by an application of Theorems \ref{t2.8} and \ref{t2.9} (resp., \eqref{2.61}).

\subsection{On singularly perturbed Dirichlet Laplacians}  \lb{s3.4} 
In this subsection we discuss an applicaton of the abstract results in \cite{Fi19} to singularly perturbed Dirichlet Laplacians. 
	
Let $\Omega\subseteq \bbR^n$ be an open bounded domain with smooth boundary $\partial\Omega$. We  will employ the usual $(n-1)$-dimensional surface measure $d \omega^{n-1}$ on $\partial\Omega$ and introduce the Dirichlet and Neumann traces on the boundary $\partial \Omega$ following Grubb \cite{Gr83} (see also the summary in \cite[Example~3.5]{AGMST10}) as follows: Consider
\begin{equation}
\dot \gamma_k \colon \begin{cases} C^{\infty}(\Omega) \to C^{\infty}(\partial\Omega), \\
u \mapsto \big(\partial_n^k\big)\big|_{\partial\Omega}, 
\end{cases} \quad k=0,1, 
\end{equation}
with $\partial_n$ denoting the interior normal derivative. By continuity, $\dot \gamma_k$ extends to a bounded operator,
\begin{equation}
\gamma_k \colon H^s(\Omega) \to H^{s-k-(1/2)}(\Omega), \quad s \in (k+(1/2),\infty),
\end{equation}
such that the map
\begin{equation}
\gamma^{(1)} = (\gamma_0,\gamma_1) \colon H^s(\Omega) \to H^{s-(1/2)}(\Omega) \times 
H^{s - (3/2)}(\Omega), \quad s \in (3/2,\infty), 
\end{equation}
satisfies 
\begin{equation}
\ker\big(\gamma^{(1)}\big) = H^s_0(\Omega), \quad \ran\big(\gamma^{(1)}\big) = 
H^{s-(1/2)}(\Omega) \times 
H^{s - (3/2)}(\Omega), \quad s \in (3/2,\infty). 
\end{equation}
In the following we employ the more suggestive notation
\begin{equation}
\gamma_D = \gamma_0, \quad \gamma_N = \gamma_1
\end{equation}
for the Dirichlet and Neumann trace operators.

Introducing the minimal Laplace operator in $L^2(\Omega)$ by 
\begin{align}
& T_{\Omega,min} f = -\Delta f,   \\
& f \in \dom(T_{\Omega,min}) = \big\{g \in H^2(\Omega) \, \big| \, \gamma_Dg = \gamma_Ng = 0\big\},
\end{align}
the corresponding maximal operator $T_{\Omega,max}$ is given by
\begin{align}
& T_{\Omega,max} f = -\Delta f,   \\
& f \in \dom(T_{\Omega,max}) = \big\{g \in L^2(\Omega) \, \big| \, \Delta g \in L^2(\Omega)\big\},
\end{align}
and one infers
\begin{equation}
T_{\Omega,min}^*=T_{\Omega,max}, \quad T_{\Omega,max}^*=T_{\Omega,min}. 
\end{equation}
Since the open set $\Omega \subset \bbR^n$ was assumed to be bounded, $T_{\Omega,min}$ is a strictly positive symmetric operator in $L^2(\Omega)$, whose Friedrichs extension is given by the self-adjoint Dirichlet realization $T_{\Omega,D}$ of the Laplacian, 
\begin{align}
& T_{\Omega,D} f = -\Delta f,   \\
& f \in \dom(T_{\Omega,D}) = \big\{g \in H^2(\Omega) \, \big| \, \gamma_Dg = 0\big\},
\end{align}
In \cite[Example~6.1]{Fi19}, the following restrictions of $T_{\Omega,D}$ are constructed: Fix 
\begin{equation}
h\in L^2(\Omega), \quad k \in C(\partial \Omega) \backslash \{0\},     \lb{3.111} 
\end{equation} 
and define
\begin{align}
\begin{split} 
& T_{\Omega,h,k} f = -\Delta f,     \lb{3.112} \\
& f \in \dom(T_{\Omega,h,k}) = \{g \in \dom(T_{\Omega,D})) \, | \, (h,g)_{L^2(\Omega)} = (k,\gamma_Ng)_{L^2(\partial\Omega)}\}.
\end{split} 
\end{align}
Then $T_{\Omega,h,k}$ is closed in $L^2(\Omega)$ by \cite[Theorem~2.11]{Fi19}. 
For convenience, and without loss of generality, we employ the normalization 
\begin{equation} 
\|k\|_{L^2(\partial\Omega)}=1 
\end{equation} 
from now on (after all, one can absorb any normalizing factor for $k$ in the function $h$). Moreover, we note that 
\begin{equation} 
T_{\Omega,min} \subseteq T_{\Omega,h,k} \, \text{ if and only if $h = 0$.} 
\end{equation} 
In addition, if one temporarily only assumes that $k \in C(\partial\Omega)$, it was shown in \cite[Example~6.1]{Fi19} that 
\begin{equation} 
\text{$T_{\Omega,h,k}$ in \eqref{3.112} is densely defined if and only if $k \neq 0$,}
\end{equation} 
explaining our choice of hypothesis on $k$ in \eqref{3.111}. 

One observes that $T_{\Omega,h,k}$ as a restriction of $T_{\Omega,D}$ is also bounded from below with a strictly positive lower bound.  
	
The adjoint $T_{\Omega,h,k}^*$ in $L^2(\Omega)$ is then given by
\begin{align}
\begin{split} 
& T_{\Omega,h,k}^* f = -\Delta f + (k, \gamma_Df)_{L^2(\partial\Omega)} h,   \\
& f \in \dom(T_{\Omega,h,k}^*) = \big\{g \in H^2(\Omega) \, \big| \, \gamma_Dg = (k,\gamma_Dg)_{L^2(\partial\Omega)} k\big\}. 
\end{split} 
\end{align}

Next, let $\eta_k$ be the unique solution of the Cauchy problem
\begin{equation}
\Delta\eta_k = 0, \quad \gamma_D \eta_k = k, \, \text{ implying } \, \eta_k \in \ker(T_{max}). 
\end{equation}
From results in \cite[Example~6.1]{Fi19}, it follows that 
\begin{equation} 
\dom(T_{\Omega,h,k}^*) = \dom(T_{\Omega,D}) \dotplus {\rm lin.span}\{\eta_k\}, 
\end{equation} 
and hence, 
\begin{equation} 
\dom(T_{\Omega,h,k}^*)\subseteq\dom(T_{\Omega,max}), 
\end{equation} 
and 
\begin{equation}
T_{\Omega,h,k}^*\subseteq T_{\Omega,max} \, \text{ if and only if $h = 0$.} 
\end{equation}
Finally, we note that the deficiency indices and subspaces are given by 
\begin{equation}
n_\pm(T_{\Omega,h,k})=1, \quad \ker(T_{\Omega,h,k}^*\mp i I) 
= {\rm lin.span}\{(T_{\Omega,D} \mp i I)^{-1}(h \mp i\eta_k)-\eta_k\}.    \lb{3.121} 
\end{equation}

Using von Neumann's theory of self-adjoint extensions of a closed, symmetric operator, the following description of all self-adjoint extensions of $T_{\Omega,h,k}$ (by \eqref{3.121} necessarily a real one-parameter family of self-adjoint extensions) was obtained in \cite[Example~6.1]{Fi19}, 
\begin{align}
& T_{\Omega,h,k,\alpha} f =T_{\Omega,h,k}^* f, \quad \alpha \in \bbR,    \no \\
& f \in \dom(T_{\Omega,h,k,\alpha})= \{g \in \dom(T_{\Omega,h,k}^*) \, | \, (k,\gamma_N g)_{L^2(\partial\Omega)}
- (h, f)_{L^2(\Omega)}    \\
& \hspace*{7.5cm} = \alpha (k,\gamma_D g)_{L^2(\partial\Omega} \},    \no \\
& T_{\Omega,h,k,\infty} = T_{\Omega,D}, \quad \alpha = \infty. 
 \end{align}
	
Next, we determine the Friedrichs and Krein--von Neumann extensions of $T_{\Omega,h,k}$ and also all nonnegative self-adjoint extensions of $T_{\Omega,h,k}$. 

We start with the the Krein--von Neumann extension $T_{\Omega,h,k,K}$ of $T_{\Omega,h,k}$. Since $T_{\Omega,h,k}$ is strictly positive, the domain $\dom(T_{\Omega,h,k,K})$ is given by
\begin{equation}
\dom(T_{\Omega,h,k,K})=\dom(T_{\Omega,h,k}) \dotplus \ker(T_{\Omega,h,k}^*),
\end{equation}
and hence it remains to determine the one-dimensional kernel of $T_{\Omega,h,k}^*$. From a straightforward calculation, it follows that 
\begin{align}
\begin{split}
T_{\Omega,h,k}^* \big((T_{\Omega,D})^{-1} h-\eta_k\big) &= - \Delta\big((T_{\Omega,D})^{-1} h - \eta_k\big) 
+ \big(h, \gamma_D \big((T_{\Omega,D})^{-1} h - k\big)\big)_{L^2(\partial\Omega} h    \\
& = \big[1 - \|k\|_{L^2(\partial\Omega)}^2 \big] h = 0.
\end{split} 
\end{align}
Thus, the operator $T_{\Omega,h,k,K}$ is given by
\begin{align}
\begin{split} 
& T_{\Omega,h,k,K} f = T_{\Omega,h,k}^* f,    \\
& f \in \dom(T_{\Omega,h,k,K})=\dom(T_{\Omega,h,k}) \dotplus {\rm lin.span}\big\{(T_{\Omega,D})^{-1} h - \eta_k\big\}. 
\end{split} 
\end{align}

Next, we turn to the Friedrichs extension $T_{\Omega,h,k,F}$ of $T_{\Omega,h,k}$. In this context we will employ the following general result in the theory of nonnegative self-adjoint extensions of a given closed, strictly positive, symmetric operator $S$ in the complex, separable Hilbert space $\cH$ to the following effect:  
\begin{align} 
\begin{split} 
& \text{If $\wti S$ is a nonnegative self-adjoint extension of $S$, then}      \lb{3.124} \\
& \quad \text{$\wti S = S_F$ if and only if $\dom\Big(\big(\wti S\big)^{1/2}\Big) \cap \ker(S^*) = \{0\}$.} 
\end{split} 
\end{align}
An application of \eqref{3.124} will show that $T_{\Omega,h,k,F}=T_{\Omega,D}$ as follows: 
First, one observes that by construction, $T_{\Omega,D}$ is a strictly positive self-adjoint extension of $T_{\Omega,h,k}$, with form domain given by $\dom\big((T_{\Omega,D})^{1/2}\big)=H_0^1(\Omega)$. Next, since $\ker(T_{\Omega,h,k}^*) = {\rm lin.span} \big\{(T_{\Omega,D})^{-1} h-\eta_k\big\}$ and 
$\gamma_D \big((T_{\Omega,D})^{-1} h - \eta_k\big) = - k \neq 0$, one concludes that $(T_{\Omega,D})^{-1} h-\eta_k \notin H^1_0(\Omega)$. Consequently,  
\begin{equation}
\dom\big((T_{\Omega,D})^{1/2}\big) \cap \ker(T_{\Omega,h,k}^*) = H^1_0(\Omega) \cap 
{\rm lin.span}\big\{(T_{\Omega,D})^{-1} h - \eta_k\big\} = \{0\},
\end{equation}
implying $T_{\Omega,D} = T_{\Omega,h,k,F}$, independently of $h \in L^2(\Omega)$ and $g \in C(\partial\Omega)$. 

We conclude these observations by next determining all nonnegative self-adjoint extensions of $T_{\Omega,h,k}$. The largest is, of course, the Friedrichs extension $T_{\Omega,h,k,F}=T_{\Omega,D}$. By the Birman--Krein--Vishik-theory (see, e.g., \cite{AS80} and the references cited therein), all others are parametrized by a nonnegative parameter $\beta \in [0,\infty)$ and explicitly given by
\begin{align}
\begin{split} 
& T_{\Omega,h,k,\beta} f = T_{\Omega,h,k} f, \quad \beta \in [0,\infty),  \\ 
& f \in \dom(T_{\Omega,h,k,\beta}) = \dom(T_{\Omega,h,k})     \\
& \hspace*{6mm} 
\dotplus {\rm lin.span}\left\{\beta (T_{\Omega,D})^{-2}h + (T_{\Omega,D})^{-1} h - \beta (T_{\Omega,D})^{-1} \eta_k-\eta_k\right\}.  
\end{split} 
\end{align}

Given all these preparations, we can finally return to the principal aim of this subsection, namely, the application of Theorems \ref{t2.8} and \ref{t2.9} to integer powers of the operator $T_{\Omega,h,k}$. Since $n_\pm(T_{\Omega,h,k})=1$, one immediately obtains 
\begin{equation}
n_{\pm} \big((T_{\Omega,h,k})^m\big) = m, \quad m \in \bbN. 
\end{equation} 
These operators are given by
\begin{align}
& (T_{\Omega,h,k})^m f = (-\Delta)^m f, \quad m \in \bbN,     \no \\
& f \in \dom\big((T_{\Omega,h,k})^m\big) = \big\{g \in H^{2m}(\Omega) \, \big| \, \gamma_D \big((\Delta)^\ell g\big) = 0, \\ 
& \quad \quad  \big(h,(\Delta)^\ell g\big)_{L^2(\Omega)} = \big(k,\gamma_N\big(((\Delta)^\ell g \big)\big)_{L^2(\partial\Omega)}, \, \ell=0,\dots,m-1\big\}.   \no
\end{align}
	
Since $T_{\Omega,h,k}$ is strictly positive, one concludes that 
\begin{equation} 
\dim(\ker(T_{\Omega,h,k}^*))=n_{\pm}(T_{\Omega,h,k})=1, 
\end{equation} 
and hence by Theorem \ref{t2.6}, this also implies 
\begin{equation} 
\dim\big(\ker\big(((T_{\Omega,h,k})^m)^*\big)\big) = n_{\pm} \big((T_{\Omega,h,k})^m\big)=m, \quad m\in\bbN. 
\end{equation} 
Indeed, $\ker\big(((T_{\Omega,h,k})^m)^*\big)$ is given by
\begin{align} 
\ker\big(((T_{\Omega,h,k})^m)^*\big) = {\rm lin.span} \big\{(T_{\Omega,D})^{-\ell} h - (T_{\Omega,D})^{-(\ell-1)} \eta_k,\, \ell=1,\dots,m\big\}.
\end{align}

\medskip 

\noindent
{\bf Acknowledgments.} We are indebted to Andrei Martinez-Finkelshtein for very helpful discussions, in particular, he supplied us with the elementary proof of Lemma \ref{l2.11}. 

\medskip


\end{document}